\documentclass[12pt]{article}
\usepackage{latexsym,etex}
\usepackage{comment}
\usepackage{amsmath,amsthm,amsfonts,amssymb,graphicx,epsfig,latexsym,color,amsbsy}
\usepackage{pictex,dcpic}
\usepackage{epsf,enumerate}
\usepackage[title]{appendix}
\usepackage[margin=3.3cm]{geometry}

  \usepackage{tikz}
  \usepackage{tkz-euclide}

\def\asdim{\operatorname{asdim}}
\def\capdim{\operatorname{capdim}}

\newtheorem{thm}{Theorem}[section]
\newtheorem{lemma}[thm]{Lemma}
\newtheorem{cor}[thm]{Corollary}
\newtheorem{prop}[thm]{Proposition}
\newtheorem*{main}{Main Theorem}{}
{}

\numberwithin{figure}{section}

\theoremstyle{remark}
\newtheorem{example}[thm]{Example}

\newtheorem{definition}[thm]{Definition}
\newtheorem{remark}[thm]{Remark}

\newtheorem*{definition*}{Definition}
\newtheorem*{remark*}{Remark}

\def\R{{\mathbb R}}

\def\Z{{\mathbb Z}}

\def\S{{\mathcal S}}

\def\A{{\mathcal A}}
\def\S{{\mathcal S}}
\def\B{{\mathcal B}}

\def\int{\operatorname{int}}
\def\vcd{\operatorname{vcd}}

\def\A{{\mathcal A}}

\def\bs{\boldsymbol{\sigma}}

\def\scm{\overset{s}{\rightarrow}}
\def\sscm{\overset{ss}{\rightarrow}}
\def\sfcm{\overset{s}{\twoheadrightarrow}}

\date{July 19, 2018}
\title{On the asymptotic dimension of the curve complex}
\author{Mladen Bestvina and Ken Bromberg\thanks{Both
    authors gratefully acknowledge the support by the National
    Science Foundation under the grant numbers DMS-1308178 and DMS-1207873
    respectively.}}
\begin{document}

\maketitle
\abstract{We give a bound, linear in the complexity of the surface, to
  the asymptotic dimension of the curve complex as well as the
  capacity dimension of the ending lamination space.}
\tableofcontents
\section{Introduction}

Let $\Sigma$ be a closed orientable surface, possibly with punctures.
The curve complex $\mathcal C(\Sigma)$ of $\Sigma$ has played a
fundamental role in recent work on the geometry of mapping class
groups. Its hyperbolicity was established by Masur and Minsky
\cite{MM}, who also introduced many tools used to study its
geometry. In \cite{bf} Bell and Fujiwara used the notion of tight
geodesics of \cite{MM} and a finiteness theorem of Bowditch \cite{bhb}
to prove that $\mathcal C(\Sigma)$ has finite asymptotic
dimension. This fact was then used in \cite{bbf} to show that mapping
class groups have finite asymptotic dimension.

Recall that a metric space $X$ has asymptotic dimension $\leq n$
provided for every $R>0$ there exists a cover of $X$ by uniformly
bounded sets so that every metric $R$-ball in $X$ intersects at most
$n+1$ elements of the cover.

Bowditch's finiteness theorem was nonconstructive and as a result Bell
and Fujiwara were not able to derive any explicit upper bounds on the
asymptotic dimension of $\mathcal C(\Sigma)$. More recently, Richard
Webb \cite{webb} gave a constructive proof of Bowditch's theorem and
gave an explicit upper bound, exponential in the complexity of the
surface, on the asymptotic dimension of $\mathcal C(\Sigma)$.

Asymptotic dimension of any visual $\delta$-hyperbolic space $X$ is closely
related to the topology of its Gromov boundary $\partial X$.  Buyalo
\cite{buyalo} introduced the notion of the {\it capacity dimension} of a
metric space and showed that $\asdim X\leq \capdim \partial X+1$, where
$\partial X$ is equipped with a visual metric. (In the context of this
paper, capacity dimension is the same as the Assouad-Nagata
dimension). Subsequently,
Buyalo-Lebedeva \cite{buyalo-lebedeva} showed that when $X$ is a
hyperbolic group, then equality holds above, and moreover,
$\capdim\partial X=\dim \partial X$.

Klarreich \cite{Kla} identified the boundary of the curve complex with
the space $\mathcal {EL}$ of ending laminations, which is a
subquotient of the space $\mathcal{PML}$ of projective measured laminations.

In his work on the topology of the ending lamination space, Gabai
\cite{gabai} produced upper bounds on the covering dimension of
$\mathcal {EL}$: $\dim\mathcal EL\leq 4g+p-4$ if $\Sigma$ has genus
$g$ and $p>0$ punctures, and $\dim\mathcal EL\leq 4g-5$ if $\Sigma$ is
closed of genus $g$. We also note that the case of the 5 times
punctured sphere was worked out earlier by Hensel and Przytycki
\cite{hp}.

\begin{main}
$\capdim \mathcal{EL}\leq 4g+p-4$ if $p>0$ and $\capdim \mathcal{EL}\leq
  4g-5$ if $p=0$.
\end{main}

\begin{cor}
$\asdim \mathcal C(\Sigma)\leq 4g+p-3$ if $p>0$ and
  $\asdim\mathcal{C}(\Sigma)\leq 4g-4$ if $p=0$.
\end{cor} 

We note that these numbers are very close to the virtual cohomological
dimension $\vcd MCG(\Sigma)$ of the mapping class group, established
by Harer \cite{harer}: if $p=0$ then $\vcd =4g-5$, if $p>0, g>0$ then
$\vcd = 4g+p-4$ and if $g=0$, $p\geq 3$ then $\vcd=p-3$.

Behrstock, Hagen and Sisto \cite{BHS} used the Main Theorem to
establish a quadratic bound on the asymptotic dimension of mapping
class groups. It is an intriguing question whether asymptotic
dimension for these groups is strictly bigger than the virtual cohomological
dimension. There are groups, see e.g. \cite{sapir}, that have finite
cohomological but infinite asymptotic dimension. However, the authors
are not aware of examples where both are finite but not equal.

Our method is to directly construct required covers of $\mathcal {EL}$
via train track neighborhoods in $\mathcal{PML}$. Exactly such a
strategy was employed by Gabai in proving his upper bounds on covering
dimension but we will need to do extra work to gain more metric
control of the covers.  Roughly speaking, train tracks give a cell
structure on $\mathcal{PML}$ and a cell structure has a natural dual
``handle decomposition'' which gives an open cover of the space of
multiplicity bounded by the dimension of the cell structure. By making
the cell structure finer and showing that the multiplicity of the the
cover does not increase in $\mathcal{EL}$ Gabai obtains his upper
bound. Note that cells of small dimension will not contain ending
laminations which is why in both Gabai's work and ours the dimension
bound is smaller than the dimension of $\mathcal{PML}$.

To bound the capacity dimension one needs to find for any sufficiently
small $\epsilon>0$ covers that have bounded multiplicity and where all
elements have diameter bounded above by $\epsilon$ while the Lebesgue
number is bounded below by a fixed fraction of $\epsilon$. This last
property will not be satisfied by family of covers constructed by
Gabai.

The main motivation for this work is an attempt to find an alternative
proof of the finiteness of asymptotic dimension of the curve complex,
one that would generalize to the hyperbolic $Out(F_n)$-complexes and
provide an approach to proving $\asdim Out(F_n)<\infty$. The notion of
tight geodesics, used in the Bell-Fujiwara argument, does not seem to
carry over to the $Out(F_n)$-complexes, and we hope that the ideas in
this paper will
provide a new blueprint for attacking this question.

For readers familiar with train tracks we give a brief sketch of the
construction of the cover which will highlight the difficulties in our
approach. The set of laminations carried by a train track $\sigma$ is
naturally parameterized by a polyhedron $P(\sigma)$ in $\R^n$. (In
what follows we will blur the distinction between a measured
lamination and a projective measured lamination.) Note that $\sigma$
carries both ending laminations and simple closed curves. We denote
the former as $P_\infty(\sigma)$ and the latter as $S(\sigma)$. A
basepoint $*$ in $\mathcal{C}(\Sigma)$ determines visual metric $\rho$
on $\mathcal{EL}$. To estimate the visual diameter of
$P_\infty(\sigma)$ we take the curve $a \in S(\sigma)$ that is closest
to $*$ in $\mathcal{C}(\Sigma)$ and then the diameter of
$P_\infty(\sigma)$ is coarsely $A^{-d(a,*)}$ for some fixed constant
$A$.

To construct our cover we will repeatedly {\em split} train tracks
along large branches. The process of splitting $\sigma$ gives two
train tracks $\sigma_+$ and $\sigma_-$ such that $P(\sigma_+) \cup
P(\sigma_-) = P(\sigma)$ and $P(\sigma_+) \cap P(\sigma_-) = P(\tau)$
where $\tau = \sigma_+ \cap \sigma_-$ is a train track with $P(\tau)$
a co-dimension one face of both $P(\sigma_+)$ and $P(\sigma_-)$. To
start the construction we take a cell structure on $\mathcal{PML}$
determined by a finite collection of train tracks. If the visual
diameter of any of the top dimensional cells is larger than a fixed $\epsilon>0$
then we split. We continue this process and stop splitting a top
dimensional cell only when its diameter is
$\leq \epsilon$.

At any finite stage of this construction we will obtain a cell
structure on all of $\mathcal{PML}$. In particular every simple closed
curve will be carried on some train track. For example one of the cells
must contain the basepoint $*$ and therefore will have large visual
diameter. It immediately follows that we will need to split infinitely
many times to get a collection of cells that have small visual size.

At the end of the construction we will have a countable collection of
train tracks $\sigma_1, \sigma_2, \dots$ each determining a top
dimensional cell. The collection of these cells is locally finite and
covers all filling laminations.
To complete the proof we
will need to establish the following facts:
\begin{itemize}
\item (Lemma \ref{visual size}) All cells $P_\infty(\sigma_i)$ have visual diameter bounded above by $\epsilon$ and bounded below  by a fixed fraction of $\epsilon$.
\item (Proposition \ref{distance}) The cells of dimension less than $\dim \mathcal{PML}$ obtained
  by intersecting $P(\sigma_i)$ also have the form $P(\sigma)$ and if
  $P_\infty(\sigma)$ is nonempty its visual diameter is also bounded
  below by a fraction of $\epsilon$.
\item (Proposition \ref{p:process}) If $a \in S(\sigma_i)$ and $b \in S(\sigma_j)$ are curves that
  are close in $\mathcal{C}(\Sigma)$ then either
\begin{enumerate}[(i)]
\item both $a$ and $b$ are close to a curve in $S(\tau) =
  S(\sigma_i) \cap S(\sigma_j)$ where $\tau$ is a subtrack of both
  $\sigma_i$ and $\sigma_j$ or

\item both $a$ and $b$ are close to the basepoint $*$ (when compared to
  $\max\{d(*,S(\sigma_i)),d(*,S(\sigma_j))\}$).
\end{enumerate}
\end{itemize}
The key to proving the first bullet is the work of Masur-Minsky on
splitting sequences (see Theorem \ref{splitting}). The second bullet
follows from an adaptation of the work of Hamenst\"adt \cite[Lemma
  5.4]{ursula} (see Propositions \ref{ursula++} and
\ref{ursula+++}). The third bullet is the key technical advance of the
paper and is proved using a version of Sela's shortening argument. (See
Lemma \ref{2 faces new}.)

{\bf Plan of the paper.} In Section 2 we consider a subdivision process
on polyhedral cell structures in abstract. In Section 3 we review
train track theory, and prove our main technical result, Lemma \ref{2
  faces new}. In Section 4 we apply this analysis and show that the
visual size of the cover of $\mathcal {FPML}$ we produce is
controlled. In Section 5 we finish the argument by producing the
required ``handle decomposition'' from our cover and checking that it
satisfies the definition of capacity dimension. Finally, in the
appendix we prove a technical result (Corollary \ref{dense ending
  laminations}) about train tracks that is presumably known to the
experts. It was a surprise to us that there are nonorientable train
tracks that carry only orientable laminations, and large birecurrent
train tracks that do not carry filling laminations. These phenomena
are discussed in the appendix.

{\bf Acknowledgements.} The authors thank the Warwick reading group
for their careful reading of the paper and their many helpful
comments: Federica Fanoni, Nicholas Gale, Francesca
Iezzi, Ronja Kuhne, Beatrice
Pozzetti, Saul Schleimer and Katie Vokes. We also thank the referee for very useful
comments.

\section{Good cell structures}

In this section we consider abstract cell structures obtained by
successively subdividing cells in an initial cell structure. 

\subsection{Polytopes}
A {\it polytope} in a finite dimensional vector space $V\cong\R^n$ is
a finite intersection of closed half-spaces.\footnote{Some authors
  require polytopes to be compact. Our polytopes will be cones on
  compact spaces.} The {\it dimension} of a
polytope $U\subset V$ is the dimension of its affine span. A {\it
  face} of $U$ is the intersection $U\cap H$ for a hyperplane
$H\subset V$ such that $U$ is contained in one of the two closed
half-spaces of $H$. The {\it relative interior} of face is its interior as a subspace of $H$. Faces of a polytope are also polytopes, a polytope
has finitely many faces, and a face of a face is a face. The union of
proper faces of a polytope is its boundary, and the complement of the
boundary is the (relative) interior. See \cite{g} or \cite{ziegler}. 
Our main example of a polytope is the set (a cone) $V(\sigma)$ of measured
laminations carried by a train track $\sigma$ on a surface $\Sigma$. 

\subsection{Cell structures}
\begin{definition}
Let $U\subset V$ be a polytope.
A finite collection
$\mathcal C$ of subsets of $U$ which are also polytopes of various
dimensions, called {\it cells}, is a {\it cell structure on $U$} if:
\begin{enumerate}[(C1)]
\item $\bigcup_{C\in\mathcal C} C=U$,
\item when two cells intersect, their intersection is a union of cells,
\item distinct cells have disjoint relative
  interiors,
\item every face of every cell in $\mathcal C$ is a union of cells.
\end{enumerate}
\end{definition}

\begin{remark} We are really thinking about the
  filtration (into skeleta) $U^0\subset
  U^1\subset\cdots\subset U^n=U$ so that the components of
  $U^i-U^{i-1}$ are open $i$-dimensional convex polytopes whose faces
  are subcomplexes.
\end{remark}

\begin{definition}
A cell structure $\mathcal C$ on an $n$-dimensional polytope $U$ is
{\it good} if
\begin{enumerate}
\item [(C5)]for every $i<n$, every $i$-dimensional cell $C\in\mathcal C$ is
  the intersection of $i$-dimensional faces of $>i$-dimensional cells
  in $\mathcal C$ that contain $C$.
\end{enumerate}
\end{definition}

For example, a convex polygon with subdivided edges is not a good cell
structure since (C5) fails. However, starting with a convex polygon
and subdividing by line segments results in a good cell structure. See
Figure \ref{GCS}.

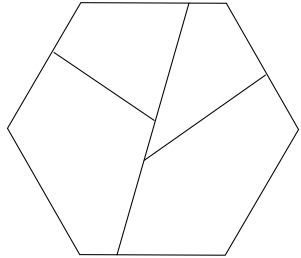
\begin{figure}[h]
\begin{center}
   \begin{tikzpicture}[y=-1cm,scale=0.6]
\sf
\draw[black] (12,12) -- (13.62444,9.21333) -- (12.02444,6.41333) -- (8.8,6.4) -- (7.17556,9.18667) -- (8.77556,11.98667) -- cycle;
\draw[black] (11.2,6.4) -- (9.6,12);
\draw[black] (8.2,7.5) -- (10.43333,9.01556);
\draw[black] (10.2,9.9) -- (12.9,8);

\end{tikzpicture}%

\caption{A hexagon subdivided 3 times results in a good cell
  structure with twelve 0-cells, fifteen 1-cells and four 2-cells.}
\label{GCS}
\end{center}
\end{figure}

\subsection{Subdivision}\label{subdividing}

Let $\mathcal C$ be a good cell structure on a polytope $U$
of dimension $n$
and let $W$ be the intersection of a co-dimension 0 cell
$\Omega\in\mathcal C$ with a hyperplane (thus we are assuming
$\dim\Omega=\dim U=n$). We will assume that the
hyperplane intersects the relative interior of $\Omega$. Construct a
new collection $\mathcal C'$ by ``cutting by $W$''. More precisely,
replace each cell $E\in \mathcal C$ which is contained in $\Omega$ and
with the property that $E-W$ is disconnected by the following three
cells: $E\cap W$ and the closures $E_1,E_2$ of the two complementary
components of $E-W$. Thus $W$ is a co-dimension 1 cell of $\mathcal
C'$. The cells $E_i$ have the same
dimension as $E$, while $\dim (E\cap W)=\dim E-1$. Figure \ref{GCS}
represents 3 consecutive subdivisions of a good cell structure
consisting of a hexagon and its faces.

\begin{lemma}\label{subdivision}
The collection $\mathcal C'$ obtained from a good cell structure
$\mathcal C$ by subdividing is a good cell
structure.
\end{lemma}

\begin{proof}
As in the notation of the definition of subdivision we subdivide a
co-dimension 0 cell $\Omega \in \mathcal C$ by a co-dimension 1 cell
$W$.  We leave it as an exercise to prove that $\mathcal C'$ is a cell
structure and argue only that it is good. We show that an $i$-cell
$C'$ of $\mathcal C'$ ($i<n$) is the intersection of $i$-faces of
$>i$-cells containing $C'$. Let $D$ be this intersection. Note that $D
\supset C'$ so we only need to show that $C'$ is not a proper subset
of $D$.

Let $C \in \mathcal C$ be the smallest cell containing $C'$. Note that
either $\dim C = \dim C' $ (and possibly $C=C'$) or $\dim C = \dim C'
+ 1$. Let $E \in \mathcal C$ be a cell that has a face $F$ that
contains $C$. Then there will be a cell $E' \subset E$ (possibly equal
to $E$) in $\mathcal C'$ with a face $F' \subset F$ and $F' \supset
C'$. By letting $E$ vary over all cells that have faces containing $C$
we see that $D \subset C$. If $C = C'$ we are now done. If not then
$C$ is disconnected by $W$ and in $\mathcal C$ becomes the 3 cells
$C_1, C_2$ and $C \cap W$ with $C'$ being one of these three
cells. Similarly, after subdivision $\Omega$ becomes three cells
$\Omega_1, \Omega_2$ and $\Omega \cap W = W$ with $C_1$ and $C_2$
contained in a faces of $\Omega_1$ and $\Omega_2$ respectively. In
particular if $C' = C_1$ (or $C'=C_2$) then $C$ is contained in a
face of $\Omega_1$ ($\Omega_2$) but that face doesn't contain any
points in $C\backslash C'$ so we must have that $C=D$. If $C' = C
\cap W$ then $C'$ is a face of $W$ but since $W$ doesn't contain any
points in $C \backslash W$ we have that $C=D$ in this case also. 
\end{proof}

\begin{remark}
When $E$ has co-dimension 1 and $U$ is a manifold (e.g. when $U$ is a
polytope), the intersection in (C5) consists of (at most) two
elements. But when the co-dimension is $>1$ the argument does not produce a
uniform bound on the number of faces required.
\end{remark}

\begin{cor}\label{bounded}
Suppose $\mathcal C$ is a good cell structure.
If a cell $E\in\mathcal C$ of dimension $i<n$ has $m$ co-dimension 1
faces, then $E$ can be written as the intersection of $\leq m$
$i$-dimensional faces of cells in $\mathcal C$ of dimension $>i$.
\end{cor}

\begin{definition}
A (finite or infinite) sequence $\mathcal C_0,\mathcal C_1,\cdots$ of
cell structures on $U$ is {\it excellent} if:
\begin{enumerate}
\item [(E1)] $\mathcal C_0$ consists of $U_p$'s and their faces,
\item [(E2)] for $i\geq 1$, $\mathcal C_i$ is obtained from $\mathcal
  C_{i-1}$ by the subdivision process along co-dimension 0 cells
  described above, or else $\mathcal C_i=\mathcal C_{i-1}$.
\end{enumerate}
\end{definition}
By Lemma \ref{subdivision}, the cell structures in an excellent
sequence are good cell structures.

\begin{remark}
Easy examples in $\R^3$ show that it is {\it not} true in general that
an $i$-cell is the intersection of $i$-faces of co-dimension 0
cells. E.g. consider the plane $x=0$ and half-planes $z=0,x\geq 0$ and
$y=0,x\leq 0$.
\end{remark}

\begin{remark}
This lemma is where our cell structure differs from Gabai's. For our
cell structure we only subdivide cells of positive co-dimension if they
are induced by subdivisions of top dimensional cells. The proposition
insures that when doing this all cells are defined via train tracks
(i.e. they are of the form $V(\theta)$
where $\theta$ is a train track, see Proposition \ref{ursula++}). Gabai
also needs this property but he achieves it by subdividing cells of
positive co-dimension.  We do not want to do this as the visual
diameter of these cells may become arbitrarily small. See Figure
\ref{codim}.
\end{remark}

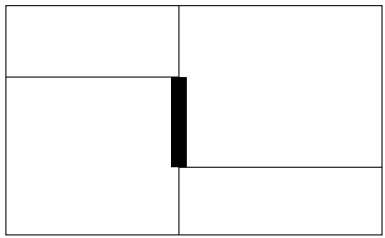
\begin{figure}[h]
  \begin{center}
    \begin{tikzpicture}[y=-1cm,scale=0.5]
\sf
\draw[black] (8.3,8.2) rectangle (18.3,14.3);
\draw[black] (8.3,10.1) -- (12.9,10.1);
\draw[black] (12.9,12.5) -- (18.3,12.5);
\draw[black] (12.9,8.2) -- (12.9,14.3);
\draw[line width=6,black] (12.9,10.1) -- (12.9,12.5);

\end{tikzpicture}%

\caption{The cell drawn in thick line arises as the intersection of
  top dimensional cells. We do not want to subdivide it further as
  this would make the visual size too small.}
\label{codim}
\end{center}
\end{figure}

\section{Train tracks}

\subsection{Notation and background}

Fix a surface $\Sigma$ of finite type. In what follows all constants will depend on the topology of $\Sigma$.
We will assume the reader is familiar with the theory of train tracks.
The standard reference is \cite{ph}. See also \cite{MM} and
\cite{ursula} for introductions to the theory. A quick definition is
that a train track in a surface $\Sigma$ is a smooth graph with a
well-defined tangent line at every point, including at the vertices,
so that no complementary component is a (smooth) disk, a monogon, a
bigon, or a punctured disk, and so that every edge can be extended in
both directions to a smoothly immersed path (these are called legal
paths or train paths). All our train tracks will always be generic (i.e. all
vertices have valence 3) and in general they will be recurrent and transversely recurrent
(birecurrent). However, there will be occasions when non-recurrent tracks will appear. A train track $\sigma\subset\Sigma$ is {\it large} if
each complementary component is homeomorphic to a disk or a once
punctured disk. A train track is {\it maximal} if all complementary
components are triangles or punctured monogons with the exception of
the punctured torus where a maximal train track contains a single
punctured bigon in its complement.

\subsubsection{Transverse measures}
The edges of the train track are {\it branches} and the vertices are
{\it switches}.  At each switch of a generic train track
$\sigma\subset \Sigma$ there are three incident half branches. Two of
these are tangent (i.e. determine the same unit tangent vector) and
are called {\it small}, while the third is a {\it large}
half branch. A branch whose both half branches are large is called
{\it large}. If both half branches are small then the branch is {\it
  small}. Otherwise the branch is {\it mixed}.

A {\it transverse measure} on a (generic) train track is an assignment
of non-negative weights to each branch that satisfy the {\it
  switch equations}. That is, at each switch the sum of the weights of
the two small half branches should be equal to the weight of the large
half branch. A transverse measure determines a unique measured
lamination on $\Sigma$. These are the laminations {\it carried} by
$\tau$.

A train track is {\it recurrent} if it admits a transverse measure
which is positive on every branch. All of our train tracks are going
to be {\it transversally recurrent} -- see \cite{ph} for the
definition. We will not use this property directly, but most results
in the literature assume it, and further there is no harm doing so as
transverse recurrence persists under splits and subtrack moves. A
train track is {\it birecurrent} if it is both recurrent and
transversally recurrent.

The set of all measured laminations on $\Sigma$ is denoted $\mathcal
{ML}$ and the set of measured laminations carried by $\sigma$ is denoted
$V(\sigma)$. Thus $V(\sigma)$ is the closed positive cone in the
vector space of real weights on the branches of $\sigma$ satisfying
the (linear) switch equations; in particular, $V(\sigma)$ is a polytope. We
denote by $\mathcal {PML}$ the projective space of measured
laminations and for a train track $\sigma$ we let
$P(\sigma)\subset\mathcal {PML}$ be the set of projective measured
laminations carried by $\sigma$. Then $P(\sigma)$ can be identified
with the projectivization of $V(\sigma)-\{0\}$. We will often blur the
distinction between a measured lamination and its projective class. 

We also denote by $\mathcal{FPML}\subset \mathcal{PML}$ the subset of
those laminations that are {\it filling}, i.e. whose complementary
components are disks or punctured disks. Given a measured lamination
$\lambda \in \mathcal{ML}$ (or $\mathcal{PML})$ we let $[\lambda]$ be
the underlying geodesic lamination.

We have a quotient map $\mathcal{FPML}\to \mathcal{EL}$ to the space
of {\it ending laminations} defined by $\lambda \mapsto
[\lambda]$. Recall that Klarreich \cite{Kla} showed that
$\mathcal{EL}$ is the Gromov boundary of the curve complex $\mathcal
C(\Sigma)$. Note that in general if $\lambda_i \in \mathcal{FPML}$ is
a sequence with limit $\lambda$ then $[\lambda]$ may be a proper
subset of the Hausdorff limit of $[\lambda_i]$.

For a train track $\sigma$ let
$P_\infty(\sigma)=P(\sigma)\cap\mathcal{FPML}$. 

At each switch the tangent direction gives a way to compare the
orientation of each branch adjacent to the switch. A train track is
{\em orientable} if each branch can be given an orientation that is
consistent at each switch.

When $\sigma$ is a generic birecurrent train track we have
$|b|/3=|v|/2=-\chi(\sigma)$, where $|b|,|v|$ denote the numbers of
branches and switches respectively.

\begin{lemma}[{\cite[Lemma 2.1.1]{ph}}]\label{dimension count}
Let $\sigma$ be a connected recurrent train track. Then the dimension of
$V(\sigma)$ is $|b|/3$ if $\sigma$ is non-orientable and $|b|/3 + 1$
if $\sigma$ is orientable.
\end{lemma} 

\begin{proof}[Sketch of the proof]
Suppose first that $\sigma$ is nonorientable. Given a switch $v$, there is
a train path that starts and ends at $v$, and the initial and terminal
half branches are the two small half branches at $v$. This path
assigns weights to the branches of $\sigma$ that satisfy all switch
equations except at $v$. This shows that the switch equations are
linearly independent, proving the assertion. 

Now suppose $\sigma$ is orientable. Choose an orientation and write
each switch equation as the sum of incoming branch(es) equals the sum of
outgoing branch(es). Then summing all switch equations yields an
identity, with each branch occurring once on both sides. Thus one
switch equation is redundant, and we need to argue that the others are
independent. Let $v,w$ be two distinct switches. Choose a train path
that connects $v$ to $w$. This path assigns weights to all edges, and
the switch equations are satisfied except at $v$ and $w$. This proves
the claim.
\end{proof}

\subsubsection{Faces of $V(\sigma)$}
There is a bijection between faces of $V(\sigma)$ and recurrent
subtracks of $\sigma$. (Here we allow train tracks to be disconnected
and to contain components that are simple closed curves.) A subtrack
of $\sigma$ may not be recurrent but any track has a unique
maximal recurrent subtrack.

\subsubsection{Splitting} 
Starting with a maximal, birecurrent train track $\sigma$ we will describe a splitting operation on train tracks that will us to subdivide $V(\sigma)$ and produce an excellent sequence of cell structures on $V(\sigma)$. We describe this now.

If $b$ is a large
branch of $\sigma$, one can produce two new train tracks
$\sigma_1,\sigma_2$ by {\it splitting} $b$. See Figure \ref{split}. We
say that $\sigma_1$ is obtained by the {\it left split} and $\sigma_2$
by the {\it right split}. 

\begin{figure}[h]
\begin{center}
   \begin{tikzpicture}[y=-1cm,scale=0.5]
\sf
\draw[black] (6.16667,8.1) -- (6.16667,14.2);
\draw[black] (7.56667,8.1) -- (7.56667,14.2);
\draw[black] (6.18,8.1) .. controls (5.90444,7.34667) .. (5,6.61333);
\draw[black] (6.17556,14.17778) .. controls (5.9,14.93111) .. (4.99556,15.66444);
\draw[black] (7.59111,8.13556) .. controls (7.86667,7.38222) .. (8.77111,6.64889);
\draw[black] (7.59111,14.19778) .. controls (7.86667,14.95111) .. (8.77111,15.68444);
\draw[black] (19.92889,8.1) -- (19.92889,14.2);
\draw[black] (21.32889,8.1) -- (21.32889,14.2);
\draw[black] (19.94222,8.1) .. controls (19.66667,7.34667) .. (18.76222,6.61333);
\draw[black] (19.93778,14.17778) .. controls (19.66222,14.93111) .. (18.75778,15.66444);
\draw[black] (21.35333,8.13556) .. controls (21.62889,7.38222) .. (22.53333,6.64889);
\draw[black] (21.35333,14.19778) .. controls (21.62889,14.95111) .. (22.53333,15.68444);
\draw[black] (13.4,8.1) -- (13.4,14.2);
\draw[black] (13.4,8.1) .. controls (13.12444,7.34667) .. (12.22,6.61333);
\draw[black] (13.39778,7.93333) .. controls (13.67333,7.18) .. (14.57778,6.44667);
\draw[black] (13.4,14.17778) .. controls (13.12444,14.93111) .. (12.22,15.66444);
\draw[black] (13.39778,14.30667) .. controls (13.67333,15.06) .. (14.57778,15.79333);
\draw[black] (6.16667,8.73333) -- (6.43333,10.2) -- (7.4,12.43333) -- (7.53333,13.26667);
\draw[black] (21.32222,8.73333) -- (21.05556,10.2) -- (20.08889,12.43333) -- (19.95556,13.26667);
\path (13.6,10.8) node[text=black,anchor=base west] {\fontsize{16.0}{19.2}\selectfont{}$e$};
\path (6.3,15.9) node[text=black,anchor=base west] {\fontsize{16.0}{19.2}\selectfont{}$\sigma_1$};
\path (20.1,15.9) node[text=black,anchor=base west] {\fontsize{16.0}{19.2}\selectfont{}$\sigma_2$};
\path (12.8,15.9) node[text=black,anchor=base west] {\fontsize{16.0}{19.2}\selectfont{}$\sigma$};

\end{tikzpicture}%

\caption{A large branch $e$ in the middle is split in two ways to give
  train tracks $\sigma_1$ and $\sigma_2$.}
\label{split}
\end{center}
\end{figure}
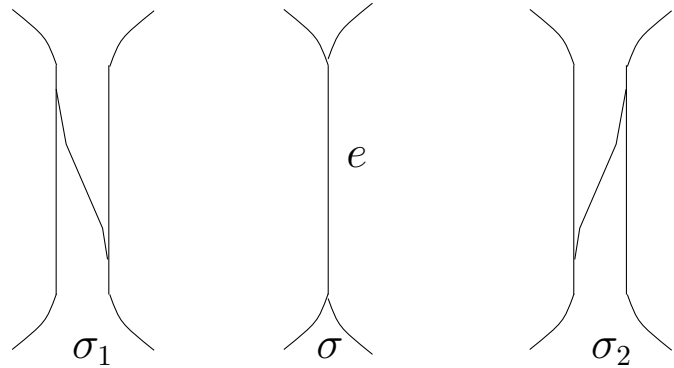

Every lamination that is carried by $\sigma$ will be carried by either
$\sigma_1$ or $\sigma_2$. If a lamination is carried by both
$\sigma_1$ and $\sigma_2$ then it will be carried by the {\em central
  split} $\tau = \sigma_1 \cap \sigma_2$, obtained from either
$\sigma_1$ or $\sigma_2$ by removing the diagonally drawn branch. 

We have the following facts:
\begin{itemize}
\item (\cite[Lemma 1.3.3(b)]{ph}) If $\sigma$ is transversely
  recurrent, so are $\sigma_1,\sigma_2$ and $\tau$.
\item (\cite[Lemma 2.1.3]{ph}) If $\sigma$ is recurrent, then either
  all three of $\sigma_1,\sigma_2,\tau$ are recurrent, or exactly one
  is recurrent.
\end{itemize}

It is also easy to see that $\sigma_1,\sigma_2$ are orientable if and
only if $\sigma$ is.

\subsubsection{Subdivision}

Now suppose $\sigma$ is a birecurrent train track and $b$ a large
branch of $\sigma$. We describe a process that subdivides
$V(\sigma)$. There are several cases. Denote by
$\sigma_1,\sigma_2,\tau$ the left, right, and central splits of $\sigma$
along $b$.

\begin{enumerate}[(S1)]
\item If all three of $\sigma_1,\sigma_2,\tau$ are recurrent, the cell
  $V(\tau)$ is a co-dimension 1 hyperplane in $V(\sigma)$ and cuts it
  into $V(\sigma_1)$ and $V(\sigma_2)$. Thus $\dim V(\sigma)=\dim
  V(\sigma_i)=\dim V(\tau)+1$. In this case we are subdividing $V(\sigma)$ as in Section \ref{subdividing}.
\item If $\sigma_1$ is recurrent but $\sigma_2$ and $\tau$ are not
  recurrent then

  $V(\sigma)=V(\sigma_1)$ while $V(\sigma_2)=V(\tau)$ will be a proper
  face of $V(\sigma)$ (possibly empty).
\item Suppose $\tau$ is recurrent, but $\sigma_1,\sigma_2$ are not. Then
  $\tau$ is the maximal recurrent subtrack of both $\sigma_i$ and
  $V(\sigma)=V(\sigma_i)=V(\tau)$. 
Since $\dim V(\sigma)=\dim V(\tau)$ Lemma \ref{dimension
  count} implies that $\sigma$ is nonorientable while $\tau$ is
  orientable. Note that
  if this case occurs every lamination carried by $\sigma$ is orientable. It
  may also happen that $\sigma$ is large while $\tau$ is not, so we
  have a situation that a large birecurrent train track does not carry
  any filling laminations. 
\end{enumerate}

\subsection{Carrying maps, stationary and active sets} 
If $\sigma$ and $\tau$ are train tracks then a map $\sigma\to\tau$ is
a {\em carrying map} if it is locally injective on each edge and takes
legal train paths to legal train paths. We also say $\sigma$ is {\em
  carried} by $\tau$ and we are implicitly assuming some explicit carrying map has been chosen.
We say that a carrying map $\sigma\to\tau$ is
   {\it fully carrying} if it is a homotopy equivalence, and we then
   write $\sigma\twoheadrightarrow \tau$. If $\lambda$ is a 
   lamination carried by $\tau$, we write $\lambda\to\tau$ for the
   carrying map. If moreover this map induces a bijection between
   complementary components that preserves the topology and numbers of sides and
   punctures, we say that $\tau$ {\it fully carries} $\lambda$ and we
   write $\lambda\twoheadrightarrow \tau$. Thus in this case splitting
   $\tau$ according to $\lambda$ always produces train tracks that
   fully carry $\lambda$.

Our definition of a track fully carrying a lamination is stronger than what is used in \cite{gabai} where it is only assumed that any realization of $\lambda$ as a measured lamination will be in the relative interior of $V(\tau)$. 

If $\sigma_1$ is a splitting of $\sigma$ there is a unique (up to
homotopy rel vertices) full carrying map $\sigma_1 \twoheadrightarrow
\sigma$ that is a bijection on vertices and is a homeomorphism outside
a small neighborhood of the large branch where the split
occurs.
If $\sigma_1$ is obtained
from a finite sequence of splittings of $\sigma$ we will always assume
that the carrying map $\sigma_1 \twoheadrightarrow \sigma$ is a
composition of such maps.

If $\tau$ is obtained from $\sigma$ by some finite combination of
splits and central splits we write $\tau \scm
\sigma$. If $\tau$ is obtained by a finite sequences of splits only
then $\tau$ is fully carried by $\sigma$ and we write $\tau \sfcm
\sigma$.
   
We also use the notation $\sigma\sscm\tau$ to mean that $\sigma$ is
obtained from $\tau$ by a sequence of splits, central splits, and
passing to subtracks. A single {\em move} is either a split, a central
split or passing to a subtrack.
The number of {\em splitting moves} in $\sigma\sscm\tau$ is
the number of splits and central splits in the
sequence. 
When we write $\sigma\sscm\tau$ we will be
implicitly assuming that some sequence of splits and subtracks has been
chosen. However, the choice of a sequence is not unique and different
choices of sequences may have a different number of moves.

Given two sequences $\sigma_1\sscm\tau$ and $\sigma_2\sscm\tau$ we would like to find a new train track $\sigma$ with $\sigma\sscm\sigma_i$ for $i=1,2$ and $V(\sigma) = V(\sigma_1)\cap V(\sigma_2)$. To accomplish this we need to develop some machinery about train tracks. The main technical result we need is Proposition \ref{ursula++}.

Given a sequence of $\sigma\sscm\tau$ we now want to define the set of
active and stationary branches. To do so we first make some general
comments about sets of branches and half branches and their
complements. Let $\S$ be a collection of branches and half branches of
a train track $\tau$ such that if $\S$ contains a branch then it
contains both half branches, and if it contains a half branch then it
contains both other half branches at the same switch.
  Then the {\em complementary branch set} $\A$
contains a branch $b$ if neither $b$ nor any of its half branches are
in $\S$ and contains a half branch $h$ if $h$ is not in $\S$. Note
that $\A$ will also have the property that if a branch is in $\A$ then
both half branches will be in $\A$ but will also have the stronger
property that if $\A$ contains both half branches of a branch then it
will contain the branch. We also note that $\S\cup\A$ may not contain
all branches of $\tau$ but it will contain all half branches.
Let $|\S|$ be the union
of branches and half branches in $\S$. We think of half branches as
 germs, so if both half branches of a branch $b$ are in $\S$ but 
 $b$ is not in $\S$ then $|\S|$ will be missing an interval in the interior of $b$.

 A convenient way to visualize the set $\S$ is to view the train track
 $\tau$ as a graph. Then switches with incident half branches in $\S$
 correspond to some vertices, and branches in $\S$ to some edges in
 $\tau$. These vertices and edges define a subgraph $\tau_S$ of $\tau$. The
 complementary set $\A$ similarly corresponds to the maximal subgraph
 of $\tau$
 disjoint from $\tau_S$. 

Given train tracks $\sigma$ and $\tau$ with $\sigma\to\tau$ a branch
$b$ in $\sigma$ is {\em stationary} if the carrying map is a
homeomorphism from a neighborhood of $b$ to its image in $\tau$. We
similarly define a half branch to be stationary and let
$\S(\sigma\sscm\tau;\sigma)$ the set of stationary branches and half
branches in $\sigma$. Note that a half branch is contained in
$\S(\sigma\sscm\tau;\sigma)$ if and only if the carrying map is a
homeomorphism on a neighborhood of the switch adjacent to the half
branch to its image. We emphasize that the stationary set depends on
the choice of carrying map and two homotopic carrying maps may have
different stationary sets. In particular a choice of sequence
$\sigma\sscm\tau$ determines the carrying map and hence the stationary
set but a different choice of sequence may determine a different
stationary set.

The image of the stationary set in $\tau$ will be a collection of branches and half branches that we denote $\S(\sigma\to\tau;\tau)$. The carrying map $\sigma\to\tau$ factors through a train track $\tau'$ if $\sigma\to\tau$ is the composition of carrying maps $\sigma\to\tau'$ and $\tau'\to\tau$ and we define $\S(\sigma\to\tau;\tau')$ to be the image of $\S(\sigma\to\tau;\sigma)$ in $\tau'$.
The main example for us is when we have a sequence $\sigma\sscm\tau$ and $\tau'$ is a track in the sequence.

The carrying map will restrict to a homeomorphism from $|\S(\sigma\to\tau;\sigma)|$ to $|\S(\sigma\to\tau;\tau)|$. However for a general carrying map the pre-image of $|\S(\sigma\to\tau;\tau)|$ in $\sigma$ may be larger than the carrying set. For carrying maps that come from sequences $\sigma\sscm\tau$ this does not happen.
\begin{lemma}\label{stationary facts}
Let $\sigma$ and $\tau$ be train tracks with $\sigma\sscm\tau$. The carrying map $\sigma\sscm\tau$ restricts to a homeomorphism from $|\S(\sigma\sscm\tau;\sigma)|$ to $|\S(\sigma\sscm\tau;\tau)|$ and the pre-image of $|\S(\sigma\sscm\tau;\tau)|$ in $\sigma$ is $|\S(\sigma\sscm\tau;\sigma)|$.
\end{lemma}
\begin{proof}
We induct on the number of moves in $\sigma\sscm\tau$. If $\sigma\sscm\tau$ is a single move then the lemma follows by direct examination. If $\sigma\sscm\tau$ has $m$ moves then we choose a train track $\tau'$ such that $\sigma\sscm\tau'\sscm\tau$ with $\sigma\sscm\tau'$ having $m-1$ moves and $\tau'\sscm\tau$ a single move. As
$$\S(\sigma\sscm\tau;\tau') =\S(\sigma\sscm\tau';\tau')\cap\S(\tau'\sscm\tau;\tau')$$
and
by the induction hypothesis the carrying map $\sigma\sscm\tau'$ restricts to a homeomorphism from $|\S(\sigma\sscm\tau;\sigma)|$ to $|\S(\sigma\sscm\tau;\tau')|$ and the carrying map $\tau'\sscm\tau$ restricts to a homeomorphism from $|\S(\sigma\sscm\tau;\tau')|$ to $|\S(\sigma\sscm\tau;\tau)|$. Therefore $\sigma\sscm\tau$ restricts to a homeomorphism from $|\S(\sigma\sscm\tau;\tau)|$ to $|\S(\sigma\sscm\tau;\tau)|$. A similar argument show that the pre-image of $|\S(\sigma\sscm\tau;\tau)|$ in $\sigma$ is $|\S(\sigma\sscm\tau;\sigma)|$.
\end{proof}

Given trains tracks $\tau_1$ and $\tau_2$ in a sequence $\sigma\sscm\tau$ and a collection of branches and half branches $\B_i\subset\S(\sigma\sscm\tau;\tau_i)$ for $i=1,2$ we write $\B_1=\B_2$ if the bijection from $\S(\sigma\sscm\tau;\tau_1)$ to $\S(\sigma\sscm\tau;\tau_2)$ takes $\B_1$ to $\B_2$.

We can define the set of {\em active branches} $\A(\sigma\sscm\tau;\tau')$ to be the complementary branch set of the stationary branches $\S(\sigma\sscm\tau;\tau')$ where $\tau'$ is a train track in the  sequence $\sigma\sscm\tau$.

Recall that in general if two half branches of a track are in the stationary set the full branch may not be. However there is one special case where this does hold.
\begin{lemma}\label{terminal stationary}
Let $\sigma$ and $\tau$ be train tracks with $\sigma\sscm\tau$. If $b$ is branch in $\tau$ such that both of its half branches are contained in $\S(\sigma\sscm\tau;\tau)$ then $b\in\S(\sigma\sscm\tau;\tau)$.
\end{lemma}

\begin{proof}
We first observe how the lemma can fail for $\S(\sigma\sscm\tau;\sigma)$. Let $b'$ be a branch in $\sigma$ with both half branches in $\S(\sigma\sscm\tau;\sigma)$. Under the carrying map $\sigma\sscm\tau$ the branch $b'$ will map to a legal path that starts and ends at a switch. (Here we are using that $\sigma\sscm\tau$ takes switches to switches by construction.) Then $b' \in \S(\sigma\sscm\tau;\sigma)$ if and only if the legal path is a single branch in $\tau$.

In our case the half branches of $b$ are in $\S(\sigma\sscm\tau;\tau)$ and as the carrying maps are good the pre-image of each will be a single half branch in $\sigma$ and therefore the pre-image of $b$ will be a single branch $b'$ in $\sigma$. Then by the above paragraph $b'\in\S(\sigma\sscm\tau;\sigma)$ and its image, $b$, will be in $\S(\sigma\sscm\tau;\tau)$. 
\end{proof}

\begin{cor}\label{disjoint}
Let $\sigma_1$, $\sigma_2$ and $\tau$ be train tracks with $\sigma_i\sscm\tau$ for $i=1,2$. Then $\A(\sigma_1\sscm\tau;\tau)\subset \S(\sigma_2\sscm\tau;\tau)$ if and only if $\A(\sigma_2\sscm\tau;\tau)\subset \S(\sigma_1\sscm\tau;\tau)$.
\end{cor}

\begin{proof}
As the set of half branches of $\tau$ is the disjoint union of the half branches in $\S(\sigma_i\sscm\tau;\tau)$ and $\A(\sigma_i\sscm\tau;\tau)$ we only need to check full branches. In particular, if $\A(\sigma_1\sscm\tau;\tau)\subset \S(\sigma_2\sscm\tau;\tau)$ and $b$ is a full branch in $\A(\sigma_2\sscm\tau;\tau)$ then we need to show that $b$ is in $\S(\sigma_1\sscm\tau;\tau)$. If $b$ is not in $\S(\sigma_1\sscm\tau;\tau)$ then by Lemma \ref{terminal stationary} a half branch $h$ of $b$ is not in $\S(\sigma_1\sscm\tau;\tau)$ and therefore $h\in\A(\sigma_1\sscm\tau;\tau)\subset \S(\sigma_2\sscm\tau;\tau)$. However, if $h\in\S(\sigma_2\sscm\tau;\tau)$ then $b\not\in\A(\sigma_2\sscm\tau;\tau)$, contradicting our assumption.
\end{proof}

We say that $\sigma_1\sscm\tau$ and $\sigma_2\sscm\tau$ are {\em disjoint} if either of the conditions of Lemma \ref{disjoint} hold.

\begin{lemma}\label{disjoint moves}
Let $\tau_n\sscm\tau_{n-1}\sscm\cdots\sscm\tau_0$ be a sequence of moves and $\sigma_0$ another train track such that $\sigma_0\sscm\tau_0$ with $\tau_n\sscm\tau_0$ and $\sigma_0\sscm\tau_0$ disjoint. Then there exists a sequence $\sigma_n\sscm\sigma_{n-1}\sscm\cdots\sscm\sigma_0$ such that:
\begin{enumerate}[(a)]
\item $\sigma_n\sscm\sigma_0$ has the same number of moves and splitting moves as $\tau_n\sscm\tau_0$;
\item $\tau_i\sscm\sigma_i$ where the sequence has the same number of moves and splitting moves as $\sigma_0\sscm\tau_0$;
\item $\A(\sigma_i\sscm\tau_i;\tau_i) \subset\S(\tau_n\sscm\tau_0;\tau_i)$ and $\A(\tau_i\sscm\tau_0;\tau_i) \subset \S(\sigma_i\sscm\tau_i;\tau_i)$;
\item $\sigma_i\sscm\tau_i$ and $\tau_n\sscm\tau_i$ are disjoint;
\item $V(\sigma_i) = V(\tau_i)\cap V(\sigma_0)$.
\end{enumerate}
\end{lemma}

\begin{proof}
Assume that $\sigma_i\sscm\tau_i$ has been constructed. We will first construct a track $\sigma_{i+1}$ with $\sigma_{i+1}\to\tau_{i+1}$ and then show that it can be realized as a sequence of moves. The move $\tau_{i+1}\sscm\tau_i$ is either a splitting move or subtrack move on a branch $b$ of $\tau_i$. As $\sigma_i\rightarrow\tau_i$ and $\tau_n\sscm\tau_i$ are disjoint we have $b\in \S(\sigma_i\rightarrow\tau_i;\tau_i)$ so the pre-image of $b$ in $\sigma_i$ is a branch $b'$ of the same type and we can perform the same move on $b'$ to form $\sigma_{i+1}$. The carrying map $\sigma_i\rightarrow\tau_i$ gives a map from $|\S(\sigma_{i+1}\sscm\sigma_i;\sigma_{i+1})|$ to $|\S(\tau_{i+1}\sscm\tau_i;\tau_{i+1})|$. If the move is a right or left split then the complement of the stationary set (for both $\sigma_{i+1}\sscm\sigma_i$ and $\tau_{i+1}\sscm\tau_i$) is the neighborhood of a small branch. If it is a central split or a subtrack move then the complement will be the interior of two branches. In all cases the map from $|\S(\sigma_{i+1}\sscm\sigma_i;\sigma_{i+1})|$ to $|\S(\tau_{i+1}\sscm\tau_i;\tau_{i+1})|$ extends to a carrying map $\sigma_{i+1}\sscm\tau_{i+1}$ that is a homeomorphism in the complement of the two stationary sets. In particular the active set $\A(\sigma_{i+1}\sscm\tau_{i+1};\tau_{i+1})$ is contained in the stationary set $\S(\tau_{i+1}\sscm\tau_i)$ and the carrying map $\tau_{i+1}\sscm\tau_i$ takes it homeomorphically to $\A(\sigma_i\sscm\tau_i;\tau_i)$. Therefore as $\A(\sigma_i\sscm\tau_i;\tau_i)\subset \S(\tau_n\sscm\tau_0;\tau_i)$ we have $\A(\sigma_{i+1}\sscm\tau_{i+1};\tau_{i+1}) \subset \S(\tau_n\sscm\tau_0;\tau_{i+1})$. The second inclusion in (c) follows from the first exactly as in Corollary \ref{disjoint}. The first inclusion in (c) implies that $\A(\sigma_{i+1}\sscm\tau_{i+1};\tau_{i+1}) \subset \S(\tau_n\sscm\tau_{i+1};\tau_{i+1})$ and therefore (d) holds.

To see that $\sigma_{i+1}\to\tau_{i+1}$ can be realized as a sequence we observe that if $\sigma_0\sscm\tau_0$ is a single move then so $\sigma_{i+1}\to\tau_{i+1}$. In general we induct on the number of moves in $\sigma_0\sscm\tau_0$.

For (e) we observe that $V(\sigma_{i+1}) \subset V(\tau_{i+1}) \cap V(\sigma_i)$. Let $\lambda$ be a lamination in $V(\tau_{i+1})\cap V(\sigma_i) \subset V(\tau_i)$. Then $\lambda$ is realized by transverse measures $m_i$, $m_{i+1}$ and $m'_i$ on $\tau_i$, $\tau_{i+1}$ and $\sigma_i$. Then $m_i$ and $m_{i+1}$ will agree on the stationary set of $\tau_{i+1}\sscm\tau_i$ and $m'_i$ and $m_i$ will agree on the stationary set of $\sigma_i\sscm\tau_i$. By examining the various cases we see that there is a transverse measure $m'_{i+1}$ on $\tau_{i+1}$ such that $m'_{i+1}$ agrees with $m'_i$ on the stationary set of $\tau_{i+1}\sscm\tau_i$ and $m'_{i+1}$ agrees with $m_{i+1}$ on the stationary set of $\sigma_{i+1}\sscm\tau_i$. For any single move transverse measures on each of the tracks that agree on the stationary set will determine the same lamination. Therefore $m'_{i+1}$ realizes $\lambda$ so $V(\sigma_{i+1}) = V(\sigma_i) \cap V(\tau_{i+1})$. As $V(\tau_{i+1}) \subset V(\tau_i)$ and $V(\sigma_i) = V(\tau_i)\cap V(\sigma_0)$ this implies that
\begin{eqnarray*}
V(\sigma_{i+1}) & = & V(\tau_{i+1}) \cap V(\sigma_i) \\
&=& V(\tau_{i+1}) \cap V(\tau_i) \cap V(\sigma_0)\\
& = & V(\tau_{i+1})\cap V(\sigma_0).
\end{eqnarray*}
\end{proof}

\begin{lemma}\label{first move}
Let $b \in \A(\sigma\scm\tau;\tau)$ be a large branch in $\tau$. Then
there exists a train track $\sigma'$ with
$\sigma'\scm\tau$ a single move on $b$ and $\sigma\scm\sigma'$ with
the sequence having at most the same number of moves as
$\sigma\scm\tau$.
\end{lemma}

\begin{proof}
Assume that the sequence $\sigma\scm\tau$ has been chosen so that the move on $b$ occurs as early as possible. More concretely, given any sequence $\sigma\scm\tau$ there exists tracks
 $\tau_1$ and $\tau_2$  in the sequence such that $\tau_1\scm\tau_2$ is a single move, $b \in \S(\tau_1\scm\tau;\tau)$ but $b \not\in \S(\tau_2\scm\tau;\tau)$. We assume that the sequence has been chosen minimizing the number of moves in $\tau_2\scm\tau$.
 
 Let $\tau_2\scm\tau_3$ be the next move in the sequence. This will be a move on a large branch $b'$ in $\tau_3$. As $b\in \S(\tau_2\scm\tau;\tau_2)$ we also have $b\in\S(\tau_2\scm\tau_3;\tau_2)$. In particular $b$ is also a large branch in $\tau_3$ and it is distinct from $b'$.  We then let $\tau'_2\scm\tau_3$ be the same move on $b$ as $\tau_1\scm\tau_2$ and note that $b'\in\S(\tau'_2\scm\tau_3;\tau_3)$ so $b'$ is a large branch in $\tau'_2$ and we can chose $\tau'_1\scm\tau'_2$ to be the same move as $\tau_2\scm\tau_3$. By direct examination we see that $\tau'_1 = \tau_1$ so we have made a new sequence $\sigma\scm\tau$ where the move on $b$ occurs earlier, a contradiction.
\end{proof}

\begin{lemma}\label{surjective}
Let $\sigma$ and $\tau$ be train tracks with $\tau$ recurrent. If $\sigma\sscm\tau$ and $V(\sigma)$ intersects the relative interior of $V(\tau)$ then $\sigma\scm\tau$.
\end{lemma}

\begin{proof}
If $V(\sigma)$ intersects the relative interior of $V(\tau)$ then the
carrying map $\sigma\sscm\tau$ must be surjective. Let
$\tau'\sscm\tau$ be the first move. This map must be surjective and
for a single move this can only happen for a split or central split.
If $\sigma\to\tau'$ is not surjective, then $\tau'\scm\tau$ is a split
and the image of $\sigma$ in $\tau'$ includes all edges except the
diagonal. Thus we can replace the first split with the central split
and proceed by induction.
\end{proof}

\begin{lemma}\label{large branches}
Let $\tau$ be a train track and $b$ a branch. Then there exists a
nonempty collection of large branches $\B$ such that if $\sigma$ is a
 train track and $\sigma\scm\tau$ with $b \in
\A(\sigma\scm\tau;\tau)$ then every branch in $\B$ is in
$\A(\sigma\scm\tau;\tau)$.
\end{lemma}

\begin{proof}
  If $b$ is large then $\B = \{b\}$. If not, consider a small half
  branch $b_1$ of $b$. There is a unique large half branch $b_2'$
  adjacent to $b_1$ and let $b_2$ be the other half branch of the
  branch $B_2$ that contains $b_2'$. If $B_2$ (i.e. $b_2$) is large then
  we note that $B_2$ must be split before $b$ becomes active. If $b_2$
  is small we continue inductively and construct half branches
  $b_3,b_4,\cdots,b_k$ ending in a large half branch $b_k$ (see
  \cite[p.127]{ph} and \cite[p.574]{ursula}) and note
  that the associated large branch $B_k$ must split before $b$ does. The
  inductive process must terminate with a large half branch for
  otherwise some half branch will repeat and by the same argument none
  of the branches listed will ever be active.
  Thus $\B$ can be taken to
  have cardinality 1 or 2.
\end{proof}

\subsection{Splitting sequences and excellent cell structures}
Given a maximal birecurrent train track $\sigma$ we describe a construction of an excellent sequence of cell structures
$\mathcal C_j$, $j=0,1,\cdots$ on the polytope $V(\sigma)$. 

We start by defining $\mathcal C_0$ to consist of $V(\sigma)$ and its
faces. Inductively, each top dimensional cell $E$ of $\mathcal C_j$
will correspond to a birecurrent track $\theta_E$ such that
$E=V(\theta_E)$. 

To define $\mathcal C_{j+1}$, choose a top dimensional cell $E$ of
$\mathcal C_j$ and a large branch $b$ of $\theta_E$. Let
$\theta_1,\theta_2,\tau$ be the left, right and central splits of
$\theta_E$ along $b$.
We now consider
the three cases (S1)-(S3) as Section 3.1.4. 

If all three $\theta_1,\theta_2,\tau$ are recurrent we split
$E=V(\theta_E)$ along the hyperplane $V(\tau)$ yielding new top
dimensional cells $V(\theta_1)$ and $V(\theta_2)$, and we subdivide
all cells that are cut by this hyperplane as described in Section
\ref{subdividing}.

If $\theta_1$ is recurrent, but $\theta_2$ and $\tau$ are not, then
$V(\theta_1) = V(\theta_E)$ and
we
define $\mathcal C_{j+1}=\mathcal C_j$ and $\theta_E=\theta_1$. We
proceed similarly if $\theta_2$ is recurrent, but $\theta_1$ and
$\tau$ are not.

The last case is when $\tau$ is recurrent, but $\theta_1,\theta_2$ are
not. However, this would imply that $\tau$ is not maximal and
therefore not a top dimensional cell by Lemma \ref{dimension count}.

A sequence $\mathcal C_j$ obtained in this way is said to be obtained
by a {\it splitting process} from $\sigma$. Note that if
$E=V(\theta_E)$ is a top dimensional cell in $\mathcal C_i$ and if
$E'=V(\theta_{E'})$ is a top dimensional cell in $\mathcal C_j$ such
that $j>i$ and $E\subsetneq E'$, then $\theta_{E'}\scm\theta_E$ and
the sequence of splits and central splits contains at most one central
split. 

We have two goals for the next few sections:
\begin{itemize}
\item We will show that every {\it every} cell of $\mathcal C_j$ has
  the form $V(\theta)$ for a suitable birecurrent train track
  $\theta$. Here the key is to show (under suitable restrictions) that
  if $\sigma_1$ and $\sigma_2$ are train tracks then $V(\sigma_1)\cap
  V(\sigma_2) = V(\sigma)$ for a train track $\sigma$. One difficulty
  is that the dimension of the intersection may be less then the
  dimension of the original cells.
\item We also need to control the ``size'' of the individual cells. We
  need to both show that for any ending lamination $\lambda\in
  V(\sigma)$ we can subdivide so that the cell containing $\lambda$ is
  small but also that the size of any proper face of cell is
  comparable to the size of the cell.
\end{itemize}
The main result we need is Proposition \ref{distance}.

\subsubsection{The curve graph and vertex cycles}
By $\mathcal C(\Sigma)$ we denote the curve graph of
$\Sigma$. Its vertices are isotopy classes of essential simple closed
curves on $\Sigma$, and two vertices are connected by an edge if the
corresponding classes have disjoint representatives. When $\Sigma$ has
low complexity $\mathcal C(\Sigma)$ can be empty or discrete, and in
the sequel we will always assume that $\mathcal C(\Sigma)$ contains
edges. In that case $\mathcal C(\Sigma)$ is connected and the
edge-path metric is $\delta$-hyperbolic \cite{MM}.

The train track
$\sigma$ carries a curve that crosses each branch at most twice, and
if it crosses a branch twice it does so with opposite
orientations. Such curves are the {\it vertex cycles} of $\sigma$.
To a
train track $\sigma\subset\Sigma$ we associate the sets
$B(\sigma)\subset \mathcal C(\Sigma)$ consisting of all vertex cycles
for $\sigma$, and the set $S(\sigma)\subset \mathcal C(\Sigma)$ of all
curves carried by $\sigma$.
We think of $B(\sigma)\subset S(\sigma)$ as a thick basepoint of
$S(\sigma)$. It is a nonempty uniformly bounded subset of $S(\sigma)$. 

\subsubsection{Splitting sequences and the geometry of the curve graph}
We begin with elementary lemma relating a single splitting to the geometry of the curve graph.

\begin{lemma}\label{3.8}
Suppose $\sigma\sscm\tau$ is a single move. Then 
$$d(B(\sigma),B(\tau))$$ is uniformly
bounded. 
\end{lemma}

\begin{proof}
Vertex cycles in subtracks are also vertex cycles in the track. In the
case of splittings, the intersection number between a vertex cycle of
$\sigma$ and a vertex cycle of $\sigma_i$ is uniformly bounded, and so is
the distance in $\mathcal C(\Sigma)$.
\end{proof}

Given a sequence $\sigma\sscm\tau$, the previous lemma implies that the corresponding sequence of vertex cycles is a coarse path in $\mathcal C(\Sigma)$. It is a theorem of Masur-Minsky that a sequence of carrying maps of birecurrent train tracks whose vertex cycles are a coarse path in $\mathcal C(\Sigma)$, then the sequence of vertex cycles are an unparameterized quasi-geodesic. In our case we know that if $\tau$ is transversely recurrent then every track in the sequence will also be transversely recurrent. However, even if $\tau$ is recurrent the other tracks in the sequence need not be so we don't automatically get a sequence of birecurrent tracks. On the other hand, for any carrying map $\sigma\to\tau$ if $\sigma$ is recurrent then its image in $\tau$ will be contained in the largest recurrent subtrack and furthermore the largest recurrent subtrack has the same vertex cycles as the original track. In particular if we replace each track in a sequence of moves with its larges recurrent subtrack we have the following:

\begin{thm}[{\cite[Theorem 1.3]{MM3}},{\cite[Theorem 1.1]{TA}}]\label{splitting}
Let $\sigma_i$ be a sequence of transversely recurrent train tracks such that $\sigma_{i+1} \sscm\sigma_i$ is a single move. Then the sequence
$B(\sigma_i)$ is a reparametrized quasi-geodesic in $\mathcal
C(\Sigma)$ with constants depending only on $\Sigma$.
\end{thm}

\begin{lemma}\label{3.3}
Let $\tau$ be a train track. Then $S(\tau)$ is quasi-convex, with
uniform constants. 
\end{lemma}

\begin{proof}
Let $a\in S(\tau)$. Split $\tau$ towards $a$. This gives a nested
sequence of tracks and thus a quasi-geodesic $g_a$ from $B(\tau)$ to $a$
that remains in $S(\tau)$.

If $a,b\in S(\tau)$ then by hyperbolicity $[a,b]$ is coarsely
contained in $g_a\cup g_b\subset S(\tau)$.
\end{proof}

The proof of the following lemma uses a technical result (Corollary \ref{dense ending laminations}) whose proof is deferred to the appendix.
\begin{lemma}\label{3.4}
Assume that $P_\infty(\tau)\neq\emptyset$. Then $S(\tau)$ is
the coarse convex hull of the set of ending laminations carried by
$\tau$.
\end{lemma}

\begin{proof}
As $\mathcal C(\Sigma)$ is hyperbolic any quasi-convex subset contains
the coarse convex hull of its Gromov boundary. By Klarreich's Theorem the
Gromov boundary of $\mathcal C(\Sigma)$ is the space of ending
laminations. If $\gamma_i \in S(\tau)$ converge to the boundary then
there exists $\lambda_i \in P(\tau)$ with $\lambda_i \to \lambda \in
P(\tau)$ such that $\gamma_i$ is a component of $[\lambda_i]$ and the
Hausdorff limit of the $\gamma_i$ contains the ending lamination
$[\lambda]$. In particular the Gromov boundary of $S(\tau)$ is exactly
the ending laminations in $P_\infty(\tau)$ so $S(\tau)$ coarsely
contains its convex hull.
 
 By Corollary \ref{dense ending laminations} for any $a \in S(\tau)$
 either $a$ is uniformly close to $B(\tau)$ or there exists a sequence
 of ending laminations $\lambda_i \in P_\infty(\tau)$ such that the
 Hausdorff limit of $[\lambda_i]$ contains $a$. Then the projections
 of $\lambda_i$ to the curve complex of the annulus around $a$ go to
 infinity and so, by the Bounded Geodesic Image Theorem (\cite{MM2}),
 when $j>>i$ the geodesic between $[\lambda_i]$ and $[\lambda_j]$
 passes within distance one of $a$.  Therefore either $a$ is distance
 at most one from the convex hull of $S(\tau)$ or it is a bounded
 distance from $B(\tau)$. However, as $S(\tau)$ is quasi-convex, it is
 coarsely connected. Therefore $S(\tau)$ is the coarse convex hull of
 the ending laminations carried by $\tau$.
\end{proof}

\begin{lemma}\label{B}
Let $\tau$ and $\sigma$ be  birecurrent train tracks with $\sigma\sscm\tau$. Then
$B(\sigma)$ is
coarsely the closest point within $S(\sigma)$ to $B(\tau)$.
\end{lemma}

\begin{proof}
Consider a splitting sequence from $\tau$ to $\sigma$.
It determines a
quasi-geodesic from $B(\tau)$ to $B(\sigma)$. Now if $a\in S(\sigma)$
is any curve, the splitting sequence and the quasi-geodesic can be continued
until $a$ crosses every branch at most once. This extended
quasi-geodesic ends at $a$ and this proves the claim.
\end{proof}

\begin{lemma}\label{splitting cycle}
Let $\tau$ and $\sigma$ be train tracks with $\sigma\sscm\tau$. There exists a constant $C = C(\Sigma)$ such that if $\sigma\sscm\tau$ has $C$ or more moves then $\A(\sigma\sscm\tau;\tau)$ contains a vertex cycle.
\end{lemma}

\begin{proof}
There is a bound, depending only on $\Sigma$, on the number of moves
that are central splits and passing to subtracks. Therefore there will be tracks $\sigma'$ and $\tau'$ in the sequence $\sigma\sscm\tau$ with $\sigma\sscm\sigma'\scm\tau'\sscm\tau$ and $\sigma'\scm\tau'$ having as many moves as we want, provided $C$ is made large.
For each right/left
split there will be two branches that are each mapped to the union of
two branches. Similarly, for each central split there will be two branches that are mapped to the union of three branches. Therefore, by increasing the number of moves we can guarantee
that there is a branch $b$ in $\sigma'$ that is mapped to a legal path
in $\tau'$, and hence $\tau$, of arbitrary
length. Any legal path in $\tau$ that is sufficiently long will contain a
subpath that closes up and that does not cross any branch exactly
once. Thus all branches it crosses are in the active set. There is a
further subpath that closes up and crosses each branch at most
once. This gives a vertex
cycle contained in the active set.

As all
constants will only depend on $\Sigma$ this implies the lemma.
\end{proof}

\begin{lemma}\label{disjoint with bounds}
Let $\tau$, $\sigma_1$ and $\sigma_2$ be  train tracks with $\sigma_1\sscm\tau$ and $\sigma_2\sscm\tau$ disjoint sequences.
Then there exists a  train track
$\sigma\sscm\sigma_i$ for $i=1,2$ with $V(\sigma) = V(\sigma_1) \cap
V(\sigma_2)$ and both $$\min\{d(B(\sigma),B(\sigma_1)),d(B(\sigma),
B(\sigma_2))\}$$ and $$\min\{d(B(\tau),B(\sigma_1)),d(B(\tau),
B(\sigma_2))\}$$ uniformly bounded.
\end{lemma}

\begin{proof}
We apply Lemma \ref{disjoint moves} to $\sigma_1\sscm\tau$ and
$\sigma_2\sscm\tau$. In particular we have a train track $\sigma$ and
a sequence $\sigma\sscm\sigma_2$ that has the same number of moves and
splitting moves as $\sigma_1\sscm\tau$ and $V(\sigma)= V(\sigma_1)\cap
V(\sigma_2)$. Let $C$ be the constant from Lemma \ref{splitting
  cycle}. If $\sigma_1\sscm\tau$ has less than $C$ moves then the
distance bound follows from Lemma \ref{3.8}. If $\sigma_2\sscm\tau$
has less than $C$ moves we swap the roles of $\sigma_1$ and $\sigma_2$
and again the lemma follows. Therefore we can assume that both
$\sigma_1\sscm\tau$ and $\sigma_2\sscm\tau$ have at least $C$ moves.

Let $\sigma'$ and $\tau'$ be the tracks in the sequences $\sigma\sscm\sigma_2$ and $\sigma_1\sscm\tau$ that are $C$ moves from $\sigma$ and $\sigma_1$. In particular, by Lemma \ref{disjoint moves}, $\sigma'\sscm\tau'$ with the same number of moves and splitting moves as $\sigma\scm\tau$. 

To bound $d(B(\sigma_1), B(\tau))$ we observe that as
$\sigma_2\sscm\tau$ has more than $C$ moves so by Lemma \ref{splitting
  cycle} there is a vertex cycle $c$ in $|\A(\sigma_2\sscm\tau;\tau)|
\subset \tau$. As $\sigma_1\sscm\tau$ and $\sigma_2\sscm\tau$ are
disjoint it follows that $c$ is in $\S(\sigma_1\sscm\tau;\tau)$ and
hence is vertex cycle in $\sigma_1$. This gives our bound on
$d(B(\sigma_1), B(\tau))$.

$$
\begindc{\commdiag}[30]
\obj(1,3)[13]{$\sigma_1$}
\obj(2,2)[22]{$\tau'$}
\obj(3,1)[31]{$\tau$}
\obj(3,5)[35]{$\sigma$}
\obj(4,4)[44]{$\sigma'$}
\obj(5,3)[53]{$\sigma_2$}

\mor{13}{22}{}
\mor{22}{31}{}
\mor{35}{44}{}
\mor{44}{53}{}
\mor{53}{31}{}
\mor{44}{22}{}
\mor{35}{13}{}
\enddc
$$

More generally, exactly the same argument works on any diamond-shaped
diagram when the arrows represent $\geq C$ disjoint moves to show that
the distance in $\mathcal C(\Sigma)$ between the vertex cycles of the
bottom train track and the train tracks on the sides is uniformly
bounded. Using the upper diamond in the diagram plus symmetry between
$\sigma_1$ and $\sigma_2$ we
conclude that $B(\sigma_1),B(\tau'),B(\tau),B(\sigma_2),B(\sigma')$
are all within uniform distance of each other. 
Finally we observe that as $\sigma\sscm\sigma'$ is exactly
$C$ moves we have a uniform bound on $d(B(\sigma), B(\sigma'))$ by
Lemma \ref{3.8}.
\end{proof}

When $A$ is a geodesic lamination on $\Sigma$, we denote by $M(A)$ the
lamination obtained from $A$ by removing all isolated non-closed
leaves. Thus $M(A)$ consists of closed leaves and of minimal
components and it is the maximal sublamination of $A$ that supports a
transverse measure. We call $M(A)$ the {\it measurable part} of $A$.

\begin{lemma}\label{measurable parts}
Suppose sequences $a_i$, $b_i$ of closed geodesics converge to
geodesic laminations $A$, $B$ respectively in the Hausdorff
topology. Assume
\begin{enumerate}[(i)]
\item both sequences go to infinity in the curve complex $\mathcal
  C(\Sigma)$, and
\item $d(a_i,b_i)$ is uniformly bounded.
\end{enumerate}
Then $A$ and $B$ have equal measurable parts, i.e. $M(A)=M(B)$.
\end{lemma}

\begin{proof} It suffices to prove the claim when $d(a_i,b_i)\leq 1$
  for all $i$. Then $A$ and $B$ have no transverse intersections. If
  $C$ is a minimal component of $M(A)$ that does not belong to $M(B)$, then it
  does not belong to $B$ either, and so for large $i$ the curve $b_i$
  is disjoint from the subsurface supporting $C$ (which may be an annulus), contradicting (i). \end{proof}

\subsection{Train tracks for cells}
Given train tracks $\tau$, $\sigma_1$ and $\sigma_2$ with $\sigma_i\sscm\tau$ we would like to find a fourth track $\sigma$ with $V(\sigma) = V(\sigma_1)\cap V(\sigma_2)$. If all three tracks are maximal and the relative interior of  $V(\sigma_1) \cap V(\sigma_2)$ is open in $V(\tau)$ then this is due to Hamenst\"adt \cite{ursula}. We begin with two preliminary results.

\begin{lemma}\label{subtrack}
Suppose $\sigma\sscm\tau$ and $\tau'\subset\tau$ is a subtrack. Then there exists a subtrack $\sigma'$ of $\sigma$ with $V(\sigma') = V(\sigma) \cap V(\tau')$, $\sigma'\sscm\tau'$ and the number of splitting moves not exceeding the number of splitting moves in $\sigma \sscm\tau$.
\end{lemma}

\begin{proof}
We first assume that $\sigma\sscm\tau$ is a single move. The general case will follow by induction.

The intersection $V(\sigma) \cap V(\tau')$ will be a face of $V(\sigma)$ and hence there will be a subtrack $\sigma' \subset \sigma$ with $V(\sigma') = V(\sigma) \cap V(\tau')$. To show that $\sigma'\sscm\tau'$ there are several cases for each type of move in $\sigma\sscm\tau$.
\begin{enumerate}[(1)]
\item $\sigma\sscm\tau$ is a subtrack move. Then $\sigma'=\sigma\cap\tau'$ so $\sigma'$ is a subtrack of $\tau'$.

\item $\sigma\scm\tau$ is a split or central split along a large branch $b$ and $\tau'$ contains $b$ and all its adjacent branches. Then $\sigma'\scm\tau'$ is a single move on the same branch $b$.

\item $\sigma\scm\tau$ is a split along $b$ and one or more of the two
  large half branches adjacent to $b$ in $\sigma$ is not in
  $\tau'$. Then the restriction of the carrying map $\sigma\scm\tau$
  to $\sigma'$ will be a switch preserving homeomorphism so $\sigma'$
  is a subtrack of $\tau$. See Figure \ref{split subtrack}.

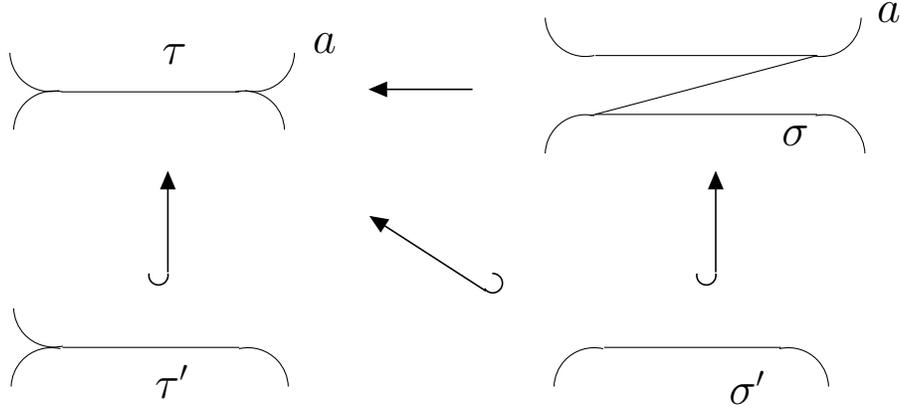
\begin{figure}[h]
\begin{center}

 \begin{tikzpicture}[scale=0.5,y=-1cm]
\sf
\draw[black] (11.49863,7.62322) +(103:1.06892) arc (103:2:1.06892);
\draw[black] (6.16803,14.42322) +(77:1.06892) arc (77:178:1.06892);
\draw[black] (6.10137,16.57678) +(-77:1.06892) arc (-77:-178:1.06892);
\draw[black] (11.33197,16.57678) +(-103:1.06892) arc (-103:-2:1.06892);
\draw[black] (25.69863,16.57678) +(-103:1.06892) arc (-103:-2:1.06892);
\draw[black] (20.5347,16.57678) +(-77:1.06892) arc (-77:-178:1.06892);
\draw[black] (6.16803,9.74344) +(-77:1.06892) arc (-77:-178:1.06892);
\draw[black] (11.23197,9.74344) +(-103:1.06892) arc (-103:-2:1.06892);
\draw[black] (6.06803,7.62322) +(77:1.06892) arc (77:178:1.06892);
\draw[black] (6.4,8.7) -- (11.1,8.7);
\draw[black] (20.30137,6.68989) +(77:1.06892) arc (77:178:1.06892);
\draw[black] (20.30137,10.37678) +(-77:1.06892) arc (-77:-178:1.06892);
\draw[black] (26.5653,6.68989) +(103:1.06892) arc (103:2:1.06892);
\draw[black] (26.6653,10.37678) +(-103:1.06892) arc (-103:-2:1.06892);
\draw[black] (26.46667,9.3) -- (20.56667,9.3);
\draw[black] (26.46667,7.73333) -- (20.56667,7.73333);
\draw[black] (26.46667,7.73333) -- (20.43333,9.33333);
\draw[semithick,black] (17.8375,13.79097) +(-93:0.25809) arc (-93:147:0.25809);
\draw[semithick,arrows=-triangle 45,black] (17.65556,14) -- (14.55556,12);
\draw[semithick,black] (8.95,13.58333) +(-11:0.25495) arc (-11:191:0.25495);
\draw[semithick,arrows=-triangle 45,black] (9.2,13.5) -- (9.2,10.8);
\draw[semithick,black] (23.51667,13.58333) +(-11:0.25495) arc (-11:191:0.25495);
\draw[semithick,arrows=-triangle 45,black] (23.76667,13.5) -- (23.76667,10.8);
\draw[semithick,arrows=-triangle 45,black] (17.3,8.63333) -- (14.53333,8.63333);
\draw[black] (11.1,15.5) -- (6.4,15.5);
\draw[black] (25.5,15.5) -- (20.8,15.5);
\path (25.23333,10.13333) node[text=black,anchor=base west] {\fontsize{16.0}{19.2}\selectfont{}$\sigma$};
\path (8.76667,7.96667) node[text=black,anchor=base west] {\fontsize{16.0}{19.2}\selectfont{}$\tau$};
\path (8.56667,16.83333) node[text=black,anchor=base west] {\fontsize{16.0}{19.2}\selectfont{}$\tau'$};
\path (12.76667,7.66667) node[text=black,anchor=base west] {\fontsize{16.0}{19.2}\selectfont{}$a$};
\path (27.76667,6.83333) node[text=black,anchor=base west] {\fontsize{16.0}{19.2}\selectfont{}$a$};
\path (23.83333,17) node[text=black,anchor=base west] {\fontsize{16.0}{19.2}\selectfont{}$\sigma'$};

\end{tikzpicture}%

\caption{If the branch $a$ is removed in $\tau'$ then it also must be removed in $\sigma'$. However, then both small branches adjacent to $a$ must be removed and $\sigma'$ will be a subtrack of $\tau$.}
\label{split subtrack}
\end{center}
\end{figure}

\item $\sigma\scm\tau$ is a central split and one or more of the half branches adjacent to $b$ in $\tau$ is not in $\tau'$. Then as in (3) $\sigma'$ is a subtrack of $\tau$.

\item $\sigma\scm\tau$ is a split and $\tau'$ contains both of the
  large half branches adjacent to $b$ in $\sigma$ and does not contain one or more
  of the two adjacent small half branches. In this case $\tau'$ is
  isotopic to a subtrack of $\sigma$ and $\sigma' = \tau'$. See Figure
  \ref{case 5}.

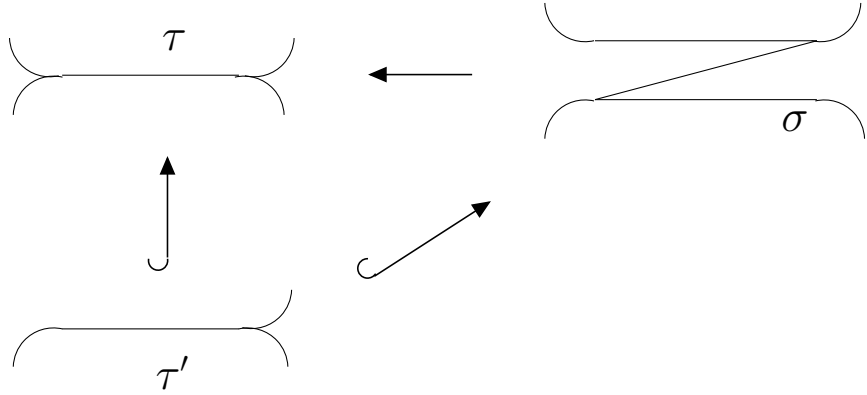
\begin{figure}[h]
\begin{center}
 \begin{tikzpicture}[scale=0.5,y=-1cm]
\sf
\draw[black] (11.49863,7.62322) +(103:1.06892) arc (103:2:1.06892);
\draw[black] (11.33197,16.44344) +(-103:1.06892) arc (-103:-2:1.06892);
\draw[black] (11.43197,14.32322) +(103:1.06892) arc (103:2:1.06892);
\draw[semithick,black] (8.95,13.58333) +(-11:0.25495) arc (-11:191:0.25495);
\draw[semithick,black] (14.51806,13.79097) +(-87:0.25809) arc (-87:-327:0.25809);
\draw[black] (6.16803,16.44344) +(-77:1.06892) arc (-77:-178:1.06892);
\draw[black] (6.16803,9.74344) +(-77:1.06892) arc (-77:-178:1.06892);
\draw[black] (11.23197,9.74344) +(-103:1.06892) arc (-103:-2:1.06892);
\draw[black] (6.06803,7.62322) +(77:1.06892) arc (77:178:1.06892);
\draw[black] (6.4,8.65) -- (11.1,8.65);
\draw[black] (20.30137,6.68989) +(77:1.06892) arc (77:178:1.06892);
\draw[black] (20.30137,10.37678) +(-77:1.06892) arc (-77:-178:1.06892);
\draw[black] (26.5653,6.68989) +(103:1.06892) arc (103:2:1.06892);
\draw[black] (26.6653,10.37678) +(-103:1.06892) arc (-103:-2:1.06892);
\draw[black] (26.46667,9.3) -- (20.56667,9.3);
\draw[black] (26.46667,7.73333) -- (20.56667,7.73333);
\draw[black] (26.46667,7.73333) -- (20.56667,9.3);
\draw[semithick,arrows=-triangle 45,black] (17.3,8.63333) -- (14.53333,8.63333);
\draw[black] (11.1,15.4) -- (6.4,15.4);
\draw[semithick,arrows=-triangle 45,black] (14.7,14) -- (17.8,12);
\draw[semithick,arrows=-triangle 45,black] (9.2,13.5) -- (9.2,10.8);
\path (25.23333,10.13333) node[text=black,anchor=base west] {\fontsize{16.0}{19.2}\selectfont{}$\sigma$};
\path (8.76667,7.96667) node[text=black,anchor=base west] {\fontsize{16.0}{19.2}\selectfont{}$\tau$};
\path (8.6,17) node[text=black,anchor=base west] {\fontsize{16.0}{19.2}\selectfont{}$\tau'$};

\end{tikzpicture}%

\caption{Case where $\tau'$ is a subtrack of $\sigma$.}
\label{case 5}
\end{center}
\end{figure}
\end{enumerate}
\end{proof}

\begin{lemma}\label{first step}
Let $\tau$, $\sigma_1$ and $\sigma_2$ be train tracks such
$\sigma_1\scm\tau$ and $\sigma_2\scm\tau$ are not disjoint. Then
there exist train tracks $\tau_1 \subset \sigma_1$,
$\tau_2\subset\sigma_2$ and $\tau'$ such that $\tau'\scm\tau$ and
$\tau_i\sscm\tau'$ with each sequence $\tau_i\sscm\tau'$ having less
splitting moves than $\sigma_i\scm\tau$. Furthermore $V(\tau_1) \cap
V(\tau_2) = V(\sigma_1) \cap V(\sigma_2)$.
\end{lemma}

$$
\begindc{\commdiag}[30]
\obj(1,2)[12]{$\sigma_1$}
\obj(1,4)[14]{$\tau_1$}
\obj(3,2)[32]{$\sigma_2$}
\obj(3,4)[34]{$\tau_2$}
\obj(2,1)[21]{$\tau$}
\obj(2,3)[23]{$\tau'$}

\mor{12}{21}{}
\mor{32}{21}{}
\mor{23}{21}{}
\mor{14}{23}{}
\mor{34}{23}{}
\mor{14}{12}{}[1,5]
\mor{34}{32}{}[1,5]
\enddc
$$

\begin{proof}
If there is branch that is active in both sequences then by Lemma
\ref{large branches} there must be a large branch $b$ that is active
in both sequences. By Lemma \ref{first move} we can assume that the
first move in both sequences is along $b$. If it is the same move then
$\tau'$ is the track obtained from this first move and $\tau_i =
\sigma_i$. If not than we let $\tau'$ be the central split on
$b$. First suppose that both $\sigma_i\to\tau$ consist of a single
move.  For at least one of them, say $\sigma_1\scm\tau$, this move
will be a right (or left) split on $b$ and $\tau'$ will be obtained
from $\sigma_1$ by removing the diagonal, and we set
$\tau_1=\tau'$. If $\sigma_2\to \tau$ is a left (or right) split we
similarly put $\tau_2=\tau'$. Finally, if $\sigma_2\to \tau$ is the
central split on $b$ we have $\tau_2=\tau'=\sigma_2$.

In general, when $\sigma_i\to\tau$ have more than one move, we use the
above paragraph for the first move and then apply
Lemma
\ref{subtrack}.

Note that in all cases $V(\tau') \supset V(\sigma_1) \cap V(\sigma_2)$ and $V(\tau_i) = V(\sigma_i) \cap V(\tau')$. It follows that $V(\tau_1) \cap V(\tau_2) = V(\sigma_1) \cap V(\sigma_2)$.
\end{proof}

\begin{prop}\label{ursula++}
Let $\tau$, $\sigma_1$ and $\sigma_2$ be train tracks such
that $\sigma_i \sscm \tau$ for $i=1,2$. Assume that $V(\sigma_1) \cap
V(\sigma_2) \neq \emptyset$.
Then there exist  train
tracks $\sigma^\pm$ and subtracks $\sigma_i'\subset\sigma_i$ such that 
\begin{enumerate}[(a)]
\item $\sigma^-\sscm\sigma_i'\sscm\sigma^+\sscm\tau$ for $i=1,2$;
\item $\sigma'_1\sscm\sigma^+$ and $\sigma'_2\sscm\sigma^+$ are disjoint;
\item
$V(\sigma^-)=V(\sigma_1)\cap V(\sigma_2)$; 
\item $\min\{d(B(\sigma_1),B(\sigma^-)),d(B(\sigma_2),B(\sigma^-))\}$
  is uniformly bounded;
  \item $\min\{d(B(\sigma_1),B(\sigma^+)),d(B(\sigma_2),B(\sigma^+))\}$
  is uniformly bounded.
\end{enumerate} 
\end{prop}

$$
\begindc{\commdiag}[30]
\obj(1,3)[13]{$\sigma_1$}
\obj(2,4)[24]{$\sigma_1'$}
\obj(3,1)[31]{$\tau$}
\obj(3,3)[32]{$\sigma^+$}
\obj(3,5)[35]{$\sigma^-$}
\obj(4,4)[44]{$\sigma_2'$}
\obj(5,3)[53]{$\sigma_2$}

\mor{13}{31}{}
\mor{53}{31}{}
\mor{32}{31}{}
\mor{24}{32}{}
\mor{44}{32}{}
\mor{35}{24}{}
\mor{35}{44}{}
\mor{24}{13}{}[1,5]
\mor{44}{53}{}[1,5]
\enddc
$$

\begin{proof}
If $\sigma_1\sscm\tau$ and $\sigma_2\sscm\tau$ are disjoint
 then the proposition follows from Lemma \ref{disjoint with
  bounds} with $\sigma^+ = \tau$ and $\sigma^-$ the track given by
Lemma \ref{disjoint with bounds}. If not we describe an algorithm that replaces $\sigma_1$ and $\sigma_2$ with subtracks $\tau_1$ and $\tau_2$ and $\tau$ with a train track $\tau'$ such that $\tau_i\sscm\tau'$ and $\tau'\sscm\tau$. 
Furthermore one of the following will hold:
\begin{enumerate}[(i)]

\item\label{dimension lower} $\dim V(\tau') <\dim V(\tau)$;

\item\label{moves lower} the number of splitting moves in $\tau_i\sscm\tau'$ is less than in $\sigma_i\sscm\tau$.
\end{enumerate}
In addition, neither the dimension nor the number of moves ever increases.
If $\tau_1\sscm\tau'$ and $\tau_2\sscm\tau'$ are disjoint then, as above, the proposition follows from Lemma \ref{disjoint with bounds}.
If \eqref{dimension lower} or \eqref{moves lower} hold we apply the algorithm to the three new tracks. As both \eqref{dimension lower} and \eqref{moves lower} can only happen a finite number of times so we must eventually have that the two sequences are disjoint.

We now describe the algorithm. Let $\tau'\subset\tau$ be the smallest birecurrent subtrack such
  that $V(\sigma_1) \cap V(\sigma_2) \subset V(\tau')$.

\begin{enumerate}[(1)]
\item If $\tau'$ is a proper subtrack of $\tau$ we let
  $\tau_i\subset\sigma_i$ be the subtracks given by Lemma
  \ref{subtrack}. In this case \eqref{dimension lower} holds.

\item If $\tau'=\tau$ then $V(\sigma_i)$ intersects the relative interior of $V(\tau)$ so by Lemma \ref{surjective} we can assume that $\sigma_i\scm\tau$. By assumption $\sigma_1\scm\tau$ and $\sigma_2\scm\tau$ are not disjoint and we replace $\sigma_1$, $\sigma_2$ and $\tau$ with the tracks $\tau_1,\tau_2$ and $\tau'$ given by Lemma \ref{first step}. In this case \eqref{moves lower} holds.
\end{enumerate}
\end{proof}

\begin{prop}\label{ursula+++}
  Assume that $\sigma\sscm\tau$ and that $\theta\scm\tau$ is a central
  split such that $V(\sigma)\cap V(\theta)$ is the intersection of
  $V(\sigma)$ with a hyperplane that intersects the relative interior
  of $V(\sigma)$. Then there is a central split $\theta'\scm\sigma$
  such that $V(\theta')=V(\sigma)\cap V(\theta)$.
\end{prop}

\begin{proof}
  Say $\theta\scm\tau$ is the central split on the large branch
  $b$. If $b$ is not in the stationary set for $\sigma\sscm\tau$ then by Lemma \ref{first move} we can assume that the first move in $\sigma\sscm\tau$ is on $b$. But then the hyperplane assumption cannot hold. Thus $b$ is in the stationary set and is a large
  branch in $\sigma$. We define $\theta$ to be the central split in
  $b$. The conclusion now follows from Lemma \ref{disjoint moves}.
\end{proof}

\begin{prop}\label{distance}
Let $\mathcal C_j$ be an excellent sequence of cell structures
obtained by splitting a train track $\tau$. To every cell
$E\in\mathcal C_j$ one can assign
a birecurrent train track $\theta_E$
satisfying the following:
\begin{enumerate}[(1)]
\item $E=V(\theta_E)$.
\item If $E$ is a top dimensional cell, then $\theta_E$ is the track
  associated to $E$ in the definition of the splitting sequence.
\item If $F\subset E$ are cells then $\theta_F\sscm
  \theta_E$.
\item There is a constant $C = C(\Sigma)$ such that for each cell $F \in \mathcal C_k$ there is a top dimensional cell $E \in \mathcal C_k$ with $F \subset E$ and $d(B(\theta_E), B(\theta_F)) \leq C$.
\end{enumerate}
In particular, if all top dimensional cells in $\mathcal C_j$ have
vertex cycles distance at most $B$ from $B(\tau)$ then $d(B(\tau),
B(\theta_E)) \leq B+C$ while if $E\subset F$ for a cell $F \in
\mathcal C_j$ with $d(B(\tau), B(\theta_F)) \ge A$ then $d(B(\tau),
B(\theta_E)) \ge A-C$.
\end{prop}

\begin{proof}
We define $\theta_E$ for $E\in\mathcal C_j$ by induction on $j$. When
$j=0$ each cell $E$ is naturally associated to a subtrack of $\tau$
and we define $\theta_E$ to be this subtrack. Now suppose that
$\theta_E$ has been defined for all cells in $\mathcal C_j$ of
dimension $>i$ for a certain $i<n$. Let $F \in \mathcal C_j$ with
$\dim F = i$. By property (C5) of an excellent sequence if $E_1,
\dots, E_\ell$ are all $i$-dimensional cells in $\mathcal C_j$ with $F
\subset E_s$ then $F = \cap E_s$. Let $F_k = E_1 \cap \dots \cap
E_k$. Via induction we have tracks $\theta_{F_k}$ with
$V(\theta_{F_k}) = F_k$ and if $E \in \mathcal C_j$ with $E_s\subset
E$ for some $s =1, \dots, k$ then $\theta_{F_k} \sscm \theta_E$. The
track $\theta_{F_k}$ is defined by applying Proposition \ref{ursula++}
to $\theta_{F_{k-1}}$ and $\theta_{E_k}$. If $\theta_F$ is not recurrent we can replace it with its largest recurrent subtrack. We then set $\theta_{F} =
\theta_{F_\ell}$ and this track will satisfy properties (1)-(3).

To get the distance bound in (4) we observe that (c) of Proposition
\ref{ursula++} gives a bound that is linear in $\ell$. While
we cannot {\it a priori} control the size of $\ell$, once we know that $F =
V(\theta_F)$ for the train track $\theta_F$ we observe that the number
of co-dimension one faces of $F$ is bounded by the number of small
branches of $\theta_F$ and hence a constant only depending on
$\Sigma$. In particular there is a subcollection of the $E_1,\dots,
E_\ell$ of uniformly bounded size whose intersection gives $F$ by
Corollary \ref{bounded}. Applying the argument of the previous
paragraph to this subcollection we get a track $\theta'_F$ with
$V(\theta'_F) = F$ and the distance bound in (4).

Finally we note that while $\theta_F$ and $\theta'_F$ may not be the
same track (and $\theta'_F$ may not satisfy (3)) since $V(\theta_F) =
V(\theta'_F)$ the two tracks have the same vertex cycles and therefore
(4) holds for $\theta_F$ also.
\end{proof}

Given a lamination $\lambda \in P_\infty(\tau)$ let $\tau_i$ be a
sequence of tracks such that $\tau_0 = \tau$, $\tau_{i+1} \sscm\tau_i$
is a single move and $\lambda \in P_\infty(\tau_i)$ for all $i$. 
We
say that the sequence is a {\em full splitting sequence} if for every
$i$ and every large branch $b$ in $\tau_i$ there exists an $i_n$ such
that $b \in \A(\tau_{i_n}\sscm\tau_i;\tau_i)$.

\begin{prop}\label{splitting sequence}
Assume that $\lambda$ is fully carried by $\tau$. Then there exists a
full splitting sequence $\tau=\tau_1,\tau_2,\cdots$ such that $\lambda$ is
fully carried by every $\tau_i$. Moreover, any infinite splitting
sequence starting at $\tau$ and carrying $\lambda$ is a full splitting
sequence. Furthermore if $\lambda'$ is carried
by every $\tau_i$ then $[\lambda] = [\lambda']$.
\end{prop}

\begin{proof}
The first statement follows from \cite[Lemma 2.1]{agol}. In fact, the
proof of \cite[Lemma 2.1]{agol} proves the stronger second
statement. The third statement is probably well known but as we could
not find a proof we provide one here. Assume that $\lambda'$ is
carried by all $\tau_i$ but $[\lambda]\neq [\lambda']$. By
\cite[Corollary 1.7.13]{ph} we can to find a birecurrent train track
$\tau'$ that carries $\lambda$, does not carry $\lambda'$ and is
carried by $\tau$. Hence it will fully carry $\lambda$, but it may not
come from a sequence of splits and central splits of $\tau$. Instead
we use \cite[Theorem 2.3.1]{ph} to find a track $\sigma$ with
$\sigma\sscm\tau'$, $\sigma\sscm\tau$ and $\lambda$ carried by
$\sigma$. As all three tracks fully carry $\lambda$ we in fact have
$\sigma\sfcm\tau'$ and $\sigma\sfcm\tau$.
 
 We will show that for sufficiently large $i$ we have $\tau_i
 \rightarrow \sigma$. As
 $\tau_i$ carries $\lambda'$ but $\sigma$ does not this will be a
 contradiction. We repeatedly apply Proposition \ref{ursula++}. Let
 $\sigma_1^+ = \tau$ and assume that we have constructed tracks
 $\sigma^+_1, \dots, \sigma^+_{j-1}$ with
 $\sigma^+_i\sfcm\sigma^+_{i-1}$, $\sigma\sfcm\sigma_i$,
 $\tau_i\sfcm\sigma^+_i$ and $\sigma\sfcm\sigma^+_i$ and $\tau_i\sfcm\sigma^+_i$ are disjoint.
As $\tau_{j}\sscm\tau_{j-1}$ we have $\tau_{j}\sfcm\sigma^+_{j-1}$ and we can apply Proposition \ref{ursula++} to $\tau_j\sscm\sigma^+_{j-1}$ and $\sigma\sfcm\sigma^+_{j-1}$ and let
 $\sigma^+_{j}=\sigma^+$ where $\sigma^+$ is as given in the proposition. Note that since $\lambda$ is
 fully carried by all of the tracks all the carrying maps given by
 Proposition \ref{ursula++} are fully carrying. This also implies that $\lambda$ is
 in the relative interior of the associated cells so we also never
 need to pass to subtracks. In particular $\sigma^+_j\sfcm\sigma^+_{j-1}$, $\sigma\sfcm\sigma_j$ and $\tau_i\sfcm\sigma^+_j$ so the induction step is complete.

When we apply Proposition \ref{ursula++}, if $\tau_j\sfcm\sigma^+_{j-1}$ and $\sigma\sfcm\sigma^+_{j-1}$ are disjoint then $\sigma^+_j = \sigma^+_{j-1}$ and $\sigma\sfcm\sigma^+_{j-1}$ and $\sigma\sfcm\sigma^+_j$ have the same number of moves. If not, then as $\tau_j\sfcm\sigma^+_{j-1}$ factors as $\tau_j\sfcm\tau_{j-1}\sfcm\sigma^+_{j-1}$ with $\tau_j\sfcm\tau_{j-1}$ a single move and $\tau_{j-1}\sfcm\sigma^+_{j-1}$ and $\sigma\sfcm\sigma^+_{j-1}$ disjoint, we have that $\sigma^+_j\sfcm\sigma^+_{j-1}$ is a single move and $\sigma\sfcm\sigma^+_j$ has one less move than $\sigma\sfcm\sigma^+_{j-1}$. This implies that the composition of sequences $\sigma\sfcm\sigma^+_j\sfcm\tau$ has the same number of moves as the original sequence $\sigma\sfcm\tau$. In particular the number of times that $\sigma^+_j\neq\sigma^+_{j-1}$ is bounded by the number of moves in $\sigma\sfcm\tau$ and there must exist an $N$ such that if $j>N$
 then $\sigma^+_j = \sigma^+_N$. The sequence is $\sigma^+_N$, $\tau_N$, $\tau_{N+1},\dots$ is a full splitting sequence so for $i$ sufficiently large $\A(\tau_i\sfcm\sigma_N;\sigma_N)$ is all of $\sigma_N$. The active branches for $\sigma\sfcm\sigma_N$ must be disjoint from $\A(\tau_i\sfcm\sigma_N;\sigma_N)$ so we must have $\sigma = \sigma_N$ and $\tau_N\sfcm\sigma$ as desired.
\end{proof}

\subsection{A shortening argument}

In this section we assume that $\sigma,\tau,\rho$ are {\it partial
  train tracks}, i.e. each is a subgraph of a train track. We allow
valence two vertices with the turn illegal, or even valence 1
vertices. Even though the main result is used only when $\tau,\rho$
are train tracks, the extra flexibility of passing to subgraphs will
make the proof easier. More precisely, we assume:

\begin{itemize}
\item
$\tau,\rho$ are two partial train tracks on $\Sigma$,
\item $\sigma$ is the graph that consists of edges that $\tau$ and
  $\rho$ have in common,
\item branches of $\tau-\sigma$ and $\rho-\sigma$ intersect
  transversally and any vertex in common to $\tau$ and $\rho$ is
  also a vertex of $\sigma$,
\item any lamination
carried by both $\tau,\rho$ is carried by $\sigma$, i.e.

\begin{equation}
\tag{*}
V(\tau)\cap
V(\rho)=V(\sigma)
\end{equation}
\end{itemize}

Given a triple $\bs = (\sigma; \tau, \rho)$ as above define the
complexity $\chi(\bs)$ to be the pair $(e(\sigma)+I(\tau,\rho),e(\sigma))$,
ordered lexicographically, where $e(\sigma)$ is the number of edges of
$\sigma$ and $I(\tau,\rho)$ is the number of transverse intersections
between the branches of $\tau$ and $\rho$. Note that for a given
complexity there are only finitely many $\bs$ up to the action of the
mapping class group.

The number of branches of $\sigma$ is uniformly bounded depending only
on the surface $\Sigma$, so the bound on $\chi(\bs)$ really only
amounts to the bound on the intersection between the branches of
$\tau$ and $\rho$.

As an example of the extra flexibility, note that if we remove an edge
of $\sigma$ from all three graphs $\sigma,\tau,\rho$ the listed
conditions continue to hold, but the new triple has smaller
complexity. In the proof below, the intersection number $I(\tau,\rho)$
will increase only if $e(\sigma)$ decreases by at least as much.

Denote by $Supp(\sigma)$ the support of $\sigma$,
i.e. the smallest subsurface that contains $\sigma$ (possibly
$\emptyset$, or disconnected, or all of $\Sigma$). Thus
$Supp(\sigma)=\emptyset$ if and only if $\sigma$ is contained in a
disk. 

\begin{lemma}\label{2 faces new}
For every $C$ and every $\chi$ there is $C'=C'(\Sigma,C,\chi)$
such that if $\chi(\bs)\leq\chi$, $a\in S(\tau)$, $b\in S(\rho)$,
$d(a,b)\leq C$, then
\begin{enumerate}[(i)] 
\item If $Supp(\sigma)=\emptyset$ then 
$$d(a,B(\tau))\leq C', d(b,B(\rho))\leq C'$$
\item If $Supp(\sigma)\neq\emptyset,\Sigma$ then 
$$d(a,\mathcal C(Supp(\sigma)))\leq C', d(b,\mathcal C(Supp(\sigma)))\leq C'$$
\item
If $Supp(\sigma)=\Sigma$ then
$$d(a,S(\sigma))\leq C', d(b,S(\sigma))\leq C'$$
\end{enumerate}
\end{lemma}

In (ii) by $\mathcal C(Supp(\sigma))$ we mean the set of curves
carried by $Supp(\sigma)$, even when $Supp(\sigma)$ is disconnected.

Most of the time when we apply Lemma \ref{2 faces new}, we will have
that $\tau,\rho$ are subtracks of some large track $\omega$ and
$\sigma=\tau\cap\rho$, and then the condition (*) is standard and
quickly follows from the fact that legal paths in the universal cover
are quasi-geodesics and that they are uniquely determined by their
endpoints on the circle at infinity.
The proof of Lemma \ref{2 faces new} is by
modifying the tracks and then $\tau,\rho$ may develop intersecting
branches.

If $\sigma$ is a train track or a partial train track, and $a$ is
carried by $\sigma$ then the {\it combinatorial length}
$\ell_\sigma(a)$ is the sum of the weights of $a$.

\begin{proof}[Proof of Lemma \ref{2 faces new}]
We will suppose such $C'$ does not exist and obtain a
contradiction.

{\bf Proof of (i).} If the lemma fails for a
particular $\bs$, there are sequences of curves $a_n\in S(\tau)$,
$b_n\in S(\rho)$ such that $d(a_n,b_n)\leq C$, $d(a_n,B(\tau))>n$,
$d(b_n,B(\rho))>n$. After passing to a subsequence $n_j$, we may
assume that $a_{n_j}\to A$, $b_{n_j}\to B$ in the Hausdorff topology,
where $A,B$ are geodesic laminations. By Lemma \ref{measurable parts}
$A$ and $B$ have the same (nonempty) measurable part $\Lambda$, which
must be carried by $\sigma$ by assumption (*). This contradicts the
assumption that $\sigma$ is contained in a disk.

{\bf Proof of (ii).} We induct on the complexity.

For each $\bs$ with $\chi(\bs)\leq \chi$ where the lemma fails, for
every $n$ there are curves $a^{\bs}_n\in S(\tau),b^{\bs}_n\in S(\rho)$
with $d(a^{\bs}_n,b^{\bs}_n)\leq C$, $d(a^{\bs}_n,\mathcal
C(Supp(\sigma)))>n$, $d(b^{\bs}_n,\mathcal C(Supp(\sigma)))>n$. We
will assume that subject to these conditions
$$\ell_{\tau}(a^{\bs}_n)+\ell_{\rho}(b^{\bs}_n)$$
is minimal possible.

To obtain a contradiction we will find a sequence of triples 
$\bs_i=(\sigma_i;\tau_i,\rho_i)$
where the lemma fails with $\bs = \bs_1$ and for each $\bs_i$ an infinite
sequence $\{n^i_j\}$ such that
\begin{enumerate}[(1)]
\item $n_j^1=j$ and $\{n^{i}_j\}$ is a subsequence of $\{n^{i-1}_j\}$
  for $i>1$;

\item $a^{\bs_i}_{n^i_j}\in S(\tau_i)$, $b^{\bs_i}_{n^i_j}\in
  S(\rho_i)$;

\item $d(a^{\bs_i}_{n^i_j},b^{\bs_i}_{n^i_j})\leq C$;

\item $d(a^{\bs_i}_{n^i_j},\mathcal C(Supp(\sigma_i)))>n^i_j$,
  $d(b^{\bs_i}_{n^i_j},\mathcal C(Supp(\sigma_i)))>n^i_j$; 

\item
  $\ell_{\tau_i}(a^{\bs_i}_{n^i_j})+\ell_{\rho_i}(b^{\bs_i}_{n^i_j}) <
  \ell_{\tau_i}(a^{\bs_{i-1}}_{n^i_j})+\ell_{\rho_i}(b^{\bs_{i-1}}_{n^i_j})$;

\item for every $i,j$,
  $\ell_{\tau_i}(a^{\bs_i}_{n^i_j})+\ell_{\rho_i}(b^{\bs_i}_{n^i_j})$
  is minimal possible subject to (2)-(4);

\item $\chi(\bs_{i}) \le \chi(\bs_{i-1})$;

\item $\bs_i$ satisfies (*), i.e. $V(\tau_i)\cap V(\rho_i)=V(\sigma_i)$.
\end{enumerate}
By (7) our sequence $\bs_i$ must eventually repeat (up to
$Mod(\Sigma)$) so there are $k< l$
with $\bs_k = \phi(\bs_l)$ for some mapping class $\phi$. By repeated
applications of (5), we have
$$\ell_{\tau_l}(a^{\bs_l}_{n^l_j})+\ell_{\rho_l}(b^{\bs_l}_{n^l_j})
  <
  \ell_{\tau_k}(a^{\bs_{k}}_{n^k_j})+\ell_{\rho_k}(b^{\bs_{k}}_{n^k_j})$$
  obtaining our contradiction to (6), since for large $j$ the curves
  $\phi(a^{\bs_l}_{n^l_j})$,$\phi(b^{\bs_l}_{n^l_j})$
  satisfy (2)-(4) (for $i=k$) and have smaller total combinatorial length than
  $a^{\bs_k}_{n^k_j}$ and $b^{\bs_k}_{n^k_j}$.

We will construct the sequence $\bs_i$ inductively.  Assume that
$\bs_i$ and the sequence $\{n^{i}_j\}$ have been defined satisfying
the above conditions. We then define a subsequence
$\{n^{i+1}_j\}$ of $\{n^{i}_j\}$ and show that there exists a
$\bs_{i+1}$ so that (1)-(8) hold with suitable choices of curves. We
first choose the subsequence $\{n^{i+1}_j\}$ such that $a^{\bs_i}_{n^{i+1}_j} \to A$
and $b^{\bs_i}_{n^{i+1}_j} \to B$ where $A,B$ are two geodesic laminations
and convergence is with respect to Hausdorff topology. The
construction of $\bs_{i+1}$ is more involved.

Lemma \ref{measurable parts} implies that $A$ and $B$ 
have the same measurable part $\Lambda$ and differ only in isolated
non-closed leaves. By assumption (8), $\Lambda$ is carried by
$\sigma_i$. Let $\bar\sigma \subset \sigma_i$ be the union of the branches
crossed by $\Lambda$. Thus $\bar\sigma$ is a train-track.

{\it Case 1.} $\bar\sigma$ has at least one illegal turn. Note that
$\Lambda$ supports a transverse measure of full support and in
particular $\bar\sigma$ has a large branch (one with maximal transverse
measure). Split along this branch so that $\Lambda$ is still
carried to obtain a new track $\sigma_{i+1}$. 

\begin{figure}[h]
  \begin{center}
    \begin{tikzpicture}[scale=0.6,y=-1cm]
\sf
\draw[black] (5.99615,10.16538) +(2:1.20434) arc (2:-75:1.20434);
\draw[black] (5.99615,7.16538) +(2:1.20434) arc (2:-75:1.20434);
\draw[black] (12.49616,7.16538) +(2:1.20434) arc (2:-75:1.20434);
\draw[black] (12.49616,10.16538) +(2:1.20434) arc (2:-75:1.20434);
\draw[black] (12.49616,10.63462) +(-2:1.20434) arc (-2:75:1.20434);
\draw[black] (8.40385,7.76538) +(178:1.20434) arc (178:255:1.20434);
\draw[black] (17.20384,7.36538) +(178:1.20434) arc (178:255:1.20434);
\draw[black] (5.99615,5.56538) +(2:1.20434) arc (2:-75:1.20434);
\draw[black] (8.40385,5.56538) +(178:1.20434) arc (178:255:1.20434);
\draw[black] (8.40385,11.93462) +(-178:1.20434) arc (-178:-255:1.20434);
\draw[black] (12.49616,11.93462) +(-2:1.20434) arc (-2:75:1.20434);
\draw[black] (17.20384,11.93462) +(-178:1.20434) arc (-178:-255:1.20434);
\draw[black] (12.49616,5.56538) +(2:1.20434) arc (2:-75:1.20434);
\draw[black] (17.20384,5.56538) +(178:1.20434) arc (178:255:1.20434);
\draw[black] (5.99615,11.93462) +(-2:1.20434) arc (-2:75:1.20434);
\draw[black] (5.99615,10.53462) +(-2:1.20434) arc (-2:75:1.20434);
\draw[black] (8.40385,11.56538) +(178:1.20434) arc (178:255:1.20434);
\draw[black] (17.20384,11.56538) +(178:1.20434) arc (178:255:1.20434);
\draw[black] (7.2,5.6) -- (7.2,11.9);
\draw[black] (13.7,5.6) -- (13.7,11.9);
\draw[black] (16,5.6) -- (16,11.9);
\draw[black] (13.7,5.5) -- (16,12.3);

\end{tikzpicture}%

\caption{}
\label{bigsplit}
\end{center}
\end{figure}

{\it Case 1a.}
The non-degenerate case that such a split is unique (i.e. the pairs of
weights at the two ends are distinct) is pictured in Figure
\ref{bigsplit}. The vertical segment represents a large branch of
$\overline\sigma$ and the two branches at the top and at the bottom
are also in $\overline\sigma$. The branches pictured on the sides are
branches of $\sigma_i-\overline\sigma$, $\tau_i-\sigma_i$, or
$\rho_i-\sigma_i$. The splitting operation consists of cutting along the
large branch thus producing two vertical branches of the split
$\overline\sigma$, adding the suitable diagonal branch so that
$\Lambda$ is carried, and attaching the side branches at exactly the
same point, to either the left or the right vertical branch. We define
$\bs_{i+1}$ to be the split version of $\bs_i$. Thus $\sigma_{i+1}$
includes the two vertical branches, the two branches at the top, the
two branches at the bottom, the diagonal branch, and any side branches
that came from $\sigma_i-\overline\sigma$. The track $\tau_{i+1}$
contains $\sigma_{i+1}$ and includes side branches that came from
$\tau_i-\overline\sigma_i$, and similarly for $\rho_{i+1}$. Observe that
$\chi(\bs_{i+1})=\chi(\bs_{i})$, so (7) holds.

{\it Claim.} For large $j$, $a_{n_j^{i+1}}\in S(\tau_{i+1})$ and
$b_{n_j^{i+1}}\in S(\rho_{i+1})$. 

Indeed, there are leaves of $\Lambda$ that cross from the upper left
[right] to the lower left [right] branch on the left diagram in Figure
\ref{bigsplit}, and likewise from upper left to lower right. The same
is therefore true for segments of $a_{n_j^{i+1}}$ for large $j$. This
prevents $a_{n_j^{i+1}}$ from entering the vertical segment say from a
side branch on the left and exiting through a side branch on the
right, or the top or bottom right branch. Since such configurations do
not occur, the Claim holds.

Thus after discarding an initial portion of each sequence, properties
(2)-(5) hold (for (4) note that $Supp(\sigma_{i+1})=Supp(\sigma_i)$
and for (5) note that since $a_{n_j^{i+1}}$ contains segments
that cross from upper left to lower left, from upper right to lower
right, and from upper left to lower right, the combinatorial length
strictly decreases after the split). Now define
$a^{\bs_{i+1}}_{n^{i+1}_j}$ and $b^{\bs_{i+1}}_{n^{i+1}_j}$ to be a
pair of curves that minimize the sum of the combinatorial lengths,
subject to (2)-(5). 

It remains to prove (8). Let $\Omega$ be a lamination carried by
$\tau_{i+1}$ and by $\rho_{i+1}$. It is therefore carried by $\tau_i$
and $\rho_i$, so by (8) for $\bs_i$ it is carried by $\sigma_i$. Now
we again have to argue that certain configurations do not occur,
e.g. that leaves of $\Omega$ do not enter on a left side branch and
exit on a right side branch. If this occurs then $\Omega$ would not be
carried by $\tau_{i+1}$ or $\rho_{i+1}$.

{\it Case 1b.}
In the degenerate case when both splits carry $\Lambda$ (i.e. when
$\Lambda$ does not cross the diagonally drawn branches in Figure
\ref{bigsplit2}), we define $\sigma_{i+1}$ to be the track obtained
from $\sigma_i$ by cutting open along the vertical segment. Thus
$\sigma_{i+1}$ does not include either of the diagonal branches. See
Figure \ref{bigsplit2}.

\begin{figure}[h]
  \begin{center}
    \begin{tikzpicture}[scale=0.6,y=-1cm]
\sf
\draw[black] (5.99615,10.16538) +(2:1.20434) arc (2:-75:1.20434);
\draw[black] (5.99615,7.16538) +(2:1.20434) arc (2:-75:1.20434);
\draw[black] (12.49616,7.16538) +(2:1.20434) arc (2:-75:1.20434);
\draw[black] (12.49616,10.16538) +(2:1.20434) arc (2:-75:1.20434);
\draw[black] (12.49616,10.63462) +(-2:1.20434) arc (-2:75:1.20434);
\draw[black] (8.40385,7.76538) +(178:1.20434) arc (178:255:1.20434);
\draw[black] (17.20384,7.36538) +(178:1.20434) arc (178:255:1.20434);
\draw[black] (5.99615,5.56538) +(2:1.20434) arc (2:-75:1.20434);
\draw[black] (8.40385,5.56538) +(178:1.20434) arc (178:255:1.20434);
\draw[black] (8.40385,11.93462) +(-178:1.20434) arc (-178:-255:1.20434);
\draw[black] (12.49616,11.93462) +(-2:1.20434) arc (-2:75:1.20434);
\draw[black] (17.20384,11.93462) +(-178:1.20434) arc (-178:-255:1.20434);
\draw[black] (12.49616,5.56538) +(2:1.20434) arc (2:-75:1.20434);
\draw[black] (17.20384,5.56538) +(178:1.20434) arc (178:255:1.20434);
\draw[black] (5.99615,11.93462) +(-2:1.20434) arc (-2:75:1.20434);
\draw[black] (5.99615,10.53462) +(-2:1.20434) arc (-2:75:1.20434);
\draw[black] (8.40385,11.56538) +(178:1.20434) arc (178:255:1.20434);
\draw[black] (17.20384,11.56538) +(178:1.20434) arc (178:255:1.20434);
\draw[black] (7.2,5.6) -- (7.2,11.9);
\draw[black] (13.7,5.6) -- (13.7,11.9);
\draw[black] (16,5.6) -- (16,11.9);
\draw[black] (13.7,5.5) -- (16,12.3);
\draw[black] (16.00667,5.33778) -- (13.70667,12.13778);

\end{tikzpicture}%

\caption{}
\label{bigsplit2}
\end{center}
\end{figure}

Next, we observe that for large $j$ the curves
$a^{\bs_{i+1}}_{n^{i+1}_j}$ cannot cross both from top left to bottom
right and from top right to bottom left, and the same is true for
$b^{\bs_{i+1}}_{n^{i+1}_j}$. Thus after passing to a further
subsequence we can add one of the two diagonal branches to
$\tau_{i+1}$ and ensure that $a^{\bs_{i+1}}_{n^{i+1}_j}\in
S(\tau_{i+1})$, and likewise $b^{\bs_{i+1}}_{n^{i+1}_j}\in
S(\rho_{i+1})$ after including one of the two diagonals. It is
possible that one diagonal is added to $\tau_{i+1}$ and the other to
$\rho_{i+1}$ and then the intersection number increases by 1. But the
number of branches of $\sigma_{i+1}$ decreased, so we still have
$\chi(\bs_{i+1})<\chi(\bs_i)$ and we are done by induction. If the
same diagonal is added to both $\tau_{i+1}$ and to $\rho_{i+1}$ we
will also add it to $\sigma_{i+1}$. The rest
of the argument is similar to the non-degenerate case.

{\it Case 2.} $\bar\sigma$ does not have any illegal turns. Thus $\bar\sigma$
is a collection of legal simple closed curves and so is $\Lambda$. In
$A$ and $B$ there must be isolated leaves spiraling towards each
component of $\Lambda$, in opposite directions on the two sides. The
spiraling directions are the same for both $A$ and $B$, since
otherwise the projection distance on the curve complex of the annulus
would be large. In other words, both $a^{\bs_{i}}_{n^{i}_j}$ and
$b^{\bs_{i}}_{n^{i}_j}$ wind around the same annulus and in the same
direction a large number of times. Applying the Dehn twist (left, or
right, as appropriate) shortens both curves as they wind around the
annulus one less time. At the same time this operation does not change
the distance to $\mathcal C(Supp(\sigma))$. This contradicts the
minimality and we are done.

{\bf Proof of (iii).} Again the proof is by induction on the
complexity. We will inductively assume (1)-(8) except that (4) is
replaced with
\begin{enumerate}
\item [(4')] $d(a^{\bs_i}_{n^i_j},S(\sigma_i))>n^i_j$,
  $d(b^{\bs_i}_{n^i_j},S(\sigma_i))>n^i_j$.
\end{enumerate}

The proof follows closely our proof of (ii). As in that proof, we pass
to a further subsequence and construct limiting laminations $A$ and
$B$ that have a common measurable part $\Lambda$ which is carried by
$\sigma_i$, and $\bar\sigma$ is the union of the edges of $\sigma$
crossed by $\Lambda$. 

There are two cases as in (ii). 

{\it Case 1.} $\bar\sigma$ contains an illegal turn. We split along a
large branch of $\bar \sigma$ as before and define $\bs_{i+1}$ in the
same way (in both subcases, whether the split is degenerate or
non-degenerate). The only change is that now we have to argue that
(4') holds, instead of (4). The reason now is that
$S(\sigma_{i+1})\subset S(\sigma_i)$.

{\it Case 2.} $\bar\sigma$ is a collection of legal loops. Now we
cannot simply apply a Dehn twist since this does not necessarily
preserve $S(\sigma_i)$. Note that there must be branches of $\sigma_i$
attached to both sides of $\bar\sigma$ for otherwise we would be in
situation (ii). 

{\it Case 2a.} All branches of $\sigma_i$ attached to a component of
$\bar\sigma$ are attached in the same direction. See Figure \ref{2a}.

\begin{figure}[h]
  \begin{center}
    \begin{tikzpicture}[scale=0.6,y=-1cm]
\sf
\draw[black] (8.2,4.6) +(-3:1.80278) arc (-3:71:1.80278);
\draw[black] (8.2,8.2) +(-3:1.80278) arc (-3:71:1.80278);
\draw[black] (11.8,6.5) +(-177:1.80278) arc (-177:-251:1.80278);
\draw[black] (11.8,9.7) +(-177:1.80278) arc (-177:-251:1.80278);
\draw[thick,black] (10,3.6) -- (10,11.2);
\draw[thick,black] (10,3.6) -- (10,11.2);

\end{tikzpicture}%

\caption{}
\label{2a}
\end{center}
\end{figure}

Thus there is a unique curve in $S(\sigma_i)$ that crosses an edge of
this component of $\bar\sigma$ (and it equals the
component). Therefore the Dehn twist preserves $S(\sigma_i)$ and we
proceed as before.

{\it Case 2b.} There are branches of $\sigma_i$ attached to a component 
of $\bar\sigma$ in opposite directions. We will assume here that every
branch of $\sigma_i$ is crossed by either $a^{\bs_i}_{n^i_j}$ or by
$b^{\bs_i}_{n^i_j}$ (or both) for every $j$, for otherwise we can
remove this edge from all three $\sigma_i,\tau_i,\rho_i$ and use
induction. 

Then we can find two branches of $\sigma_i$ attached in opposite
directions and on opposite sides of this component of $\bar\sigma$
(the curves $a^{\bs_i}_{n^i_j}$ or $b^{\bs_i}_{n^i_j}$ spiral and
cannot escape on the same side). In
other words, we have a picture as in Figure \ref{bigsplit} where the
vertical segment as well as top left and lower right branches [or top
  right and lower left branches] belong to $\bar\sigma$, and the top
right and the lower left branches [or top left and lower right
  branches] belong to $\sigma_i-\bar\sigma$. Perform the split as in
Figure \ref{bigsplit} so that $\Lambda$ is carried. If there are any
side branches attached to the vertical segment, then after the split
the number of side branches attached to $\bar\sigma$ is strictly
smaller and we may induct on this number. If there are no such side
branches, then the combinatorial lengths of $a^{\bs_i}_{n^i_j}$ and
$b^{\bs_i}_{n^i_j}$ strictly decrease after the split (e.g. consider a
piece of $a^{\bs_i}_{n^i_j}$ that enters $\bar\sigma$ through the top
branch which is not part of $\bar\sigma$). Then proceed as before, by
defining $a^{\bs_{i+1}}_{n^{i+1}_j}$ and $b^{\bs_{i+1}}_{n^{i+1}_j}$
to be curves that minimize combinatorial length subject to (2)-(4').
\end{proof}

We will only use two special cases of Lemma \ref{2
  faces new}, and we state them below.

\begin{cor}\label{subtracks}
For every $C>0$ there is $C'>0$ depending only on the surface $\Sigma$
so that the following holds.
Let $\tau$ be a large track on $\Sigma$. Assume one of the following.
\begin{enumerate}[(I)]
\item $\tau_1,\tau_2$ are large subtracks of
$\tau$. Let $\sigma=\tau_1\cap\tau_2$. After pruning dead ends,
$\sigma$ becomes a track (possibly empty) and $V(\sigma)=V(\tau_1)\cap
V(\tau_2)$. 
\item $\tau_1,\tau_2$ are the two tracks obtained
from $\tau$ by splitting along a large branch $e$, and $\sigma$ the
track obtained by a central split at $e$. Thus $P(\tau)=P(\tau_1)\cup
P(\tau_2)$ and $P(\sigma)=P(\tau_1)\cap P(\tau_2)$.
\end{enumerate}

Then one of the following holds.
\begin{enumerate}[(1)]
\item
$\sigma$ is not large (possibly it is
empty), and for any two curves $a_i\in P(\tau_i)$ with $d(a_1,a_2)\leq
C$
it follows that $d(a_i,B(\tau_i))\leq C'$, or
\item $\sigma$ is large and for any two curves $a_i\in S(\tau_i)$ 
with $d(a_1,a_2)\leq C$ 
there is a curve $c\in S(\sigma)$ such that $d(a_i,c)\leq C'$.
\end{enumerate}
\end{cor}

\section{Cell structures via splittings}

Now we take $U=V(\tau)$ for a recurrent, transversely recurrent,
maximal train track $\tau$.

Let $\mathcal C_j$ be the excellent sequence obtained by repeating
the subdivision process, at every step choosing one of the
co-dimension 0 cells $V(\sigma)$, $\sigma$ a recurrent, transversely
recurrent, maximal train track, and splitting in a selected
large branch. Thus inductively, co-dimension 0 cells in $\mathcal C_j$
are in 1-1 correspondence with a set of train tracks, each
obtained from $\tau$ by a splitting sequence, and all tracks in the
splitting sequence correspond to co-dimension 0 cells in $\mathcal C_k$
for $k\leq j$.

\subsection{Interpolating curves process}

In this section we set the groundwork for proving that the distance
between disjoint cells is not too small. This follows easily from Lemma
\ref{2 faces new} when the associated train tracks have bounded
intersection number. To handle the general case we define a certain
iterative procedure that constructs sequences of curves relating
different cells in $\mathcal C_j$.

We start by defining a
sequence $C_0,C_1,\dots$ inductively. Here $C_0>0$ is a fixed
constant, and $C_{i+1}$ is defined as $C'$ in Corollary
\ref{subtracks} for the constant $C=2^nC_i$, where $n=\dim\mathcal
    {ML}$.

When $a$ is a simple closed curve we denote by $Carr_j(a)$ the {\it
  carrier} of $a$ in $\mathcal C_j$, i.e. the smallest cell of
$\mathcal C_j$ that contains $a$.

\begin{definition}
A sequence of curves ${\bf a}=a_0,a_2,\dots,a_m$ in $\Sigma$ is {\it
  good} with respect to the cell structure $\mathcal C_j$ (or
$\mathcal C_j$-{\it good}) if for any two adjacent curves
$a_i,a_{i+1}$ in the sequence the carriers $Carr_j(a_i)$ and
$Carr_j(a_{i+1})$ are nested (or possibly equal),
i.e. $Carr_j(a_i)\subseteq Carr_j(a_{i+1})$ or
$Carr_j(a_{i+1})\subseteq Carr_j(a_i)$.
\end{definition}

A sequence which is $\mathcal C_j$-good may not be
$\mathcal C_{j+1}$-good. We now describe an inductive procedure that
consists of inserting curves to produce $\mathcal C_k$-good sequences
with $k$ large.

We start with a $\mathcal C_0$-good sequence of bounded length. For
example, we might start with a sequence $a_0,a_1$ of length 2
consisting of two curves in the interior of the same cell in $\mathcal
C_0$. Inductively assume that we inserted some curves in the sequence
and obtained a $\mathcal C_{j}$-good sequence ${\bf a}$.

Suppose $a,b$ are two consecutive curves in ${\bf a}$ that fail to
satisfy the definition of $\mathcal C_{j+1}$-good, that is, the 
carriers $A=Carr_{j+1}(a)$ and $B=Carr_{j+1}(b)$ are not
nested. There are several cases.

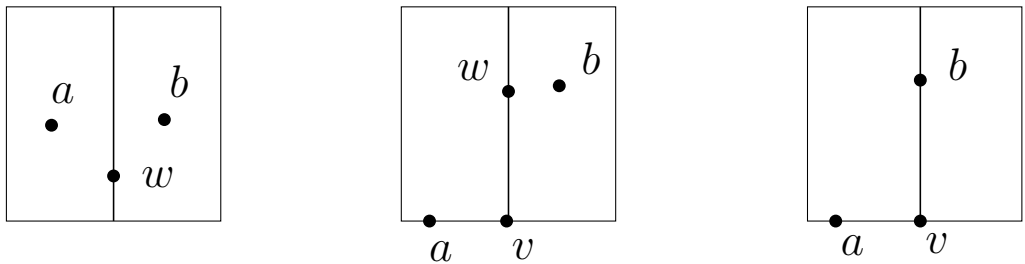
\begin{figure}[h]
  \begin{center}
    \begin{tikzpicture}[scale=0.75,y=-1cm]
\sf
\filldraw[semithick,black] (3.3,6.3) circle (0.1cm);
\filldraw[semithick,black] (4.4,7.2) circle (0.1cm);
\filldraw[semithick,black] (10,8) circle (0.1cm);
\filldraw[semithick,black] (11.4,5.7) circle (0.1cm);
\filldraw[semithick,black] (12.3,5.6) circle (0.1cm);
\filldraw[semithick,black] (17.2,8) circle (0.1cm);
\filldraw[semithick,black] (18.7,8) circle (0.1cm);
\filldraw[semithick,black] (18.7,5.5) circle (0.1cm);
\filldraw[semithick,black] (11.36667,8) circle (0.1cm);
\filldraw[semithick,black] (5.3,6.2) circle (0.1cm);
\draw[black] (6.3,8) -- (6.3,4.2) -- (2.5,4.2) -- (2.5,8) -- cycle;
\draw[black] (13.3,8) -- (13.3,4.2) -- (9.5,4.2) -- (9.5,8) -- cycle;
\draw[black] (20.5,8) -- (20.5,4.2) -- (16.7,4.2) -- (16.7,8) -- cycle;
\draw[semithick,black] (4.4,4.2) -- (4.4,8);
\draw[semithick,black] (11.4,4.2) -- (11.4,8);
\draw[semithick,black] (18.7,4.2) -- (18.7,8);
\path (3.1,5.9) node[text=black,anchor=base west] {\fontsize{16.0}{19.2}\selectfont{}$a$};
\path (5.2,5.8) node[text=black,anchor=base west] {\fontsize{16.0}{19.2}\selectfont{}$b$};
\path (4.7,7.4) node[text=black,anchor=base west] {\fontsize{16.0}{19.2}\selectfont{}$w$};
\path (9.8,8.7) node[text=black,anchor=base west] {\fontsize{16.0}{19.2}\selectfont{}$a$};
\path (12.5,5.4) node[text=black,anchor=base west] {\fontsize{16.0}{19.2}\selectfont{}$b$};
\path (10.3,5.5) node[text=black,anchor=base west] {\fontsize{16.0}{19.2}\selectfont{}$w$};
\path (17.1,8.6) node[text=black,anchor=base west] {\fontsize{16.0}{19.2}\selectfont{}$a$};
\path (19,5.5) node[text=black,anchor=base west] {\fontsize{16.0}{19.2}\selectfont{}$b$};
\path (11.26667,8.66667) node[text=black,anchor=base west] {\fontsize{16.0}{19.2}\selectfont{}$v$};
\path (18.6,8.56667) node[text=black,anchor=base west] {\fontsize{16.0}{19.2}\selectfont{}$v$};

\end{tikzpicture}%

\caption{Interpolating points to achieve goodness at the next stage.}
\label{f:subdivision}
\end{center}
\end{figure}

\begin{enumerate}[(i)]
\item $Carr_j(a)=Carr_j(b)$, we call this cell $C$. Thus the
  subdivision operation splits $C$ into $A$ and $B$ (if the cut
  contains either point, $A$ and $B$ would be nested) and $A\cap B=W$
  is the co-dimension 1 cut. See the left diagram in Figure
  \ref{f:subdivision}. We now apply Corollary \ref{subtracks}(II) to
  the train tracks $\tau_1,\tau_2,\sigma$ such that $V(\tau_1)$ and
  $V(\tau_2)$ are the two splits of $V(\tau)=C$ (so $\tau_1,\tau_2$
  have branches $e_1,e_2$ that intersect) and $\sigma$ is a common
  subtrack of $\tau_1,\tau_2$ obtained by deleting $e_1$ from $\tau_1$
  or $e_2$ from $\tau_2$ (these tracks exist by Proposition
  \ref{ursula+++}). Therefore we obtain a curve $w\in P(\sigma)=W$,
  and if $d(a,b)\leq C_i$ we have in addition that $d(a,w)\leq
  C_{i+1}$, $d(b,w)\leq C_{i+1}$. We insert $w$ in the sequence
  between $a$ and $b$. The consecutive curves in $a,w,b$ satisfy the
  $\mathcal C_{j+1}$-goodness condition.

\item $Carr_j(a)\subsetneq Carr_j(b)$. This is depicted in the other
  two diagrams in Figure \ref{f:subdivision}. Notice that the cut $W$
  cannot contain $a$, or else the goodness condition would hold in
  $\mathcal C_{j+1}$. There are two further subcases. If $b$ does not
  belong to $W$ either, we are in the situation of the middle
  diagram. First apply Corollary \ref{subtracks}(II) as in (i) above
  to find $w\in W$. Then apply Corollary \ref{subtracks}(I) to curves
  $a$ and $w$ to find a curve $v$ carried by the intersection of the
  $\mathcal C_{j+1}$-carriers of $a$ and $w$. Finally, interpolate to
  get the sequence $a,v,b$. The other subcase is that $b\in W$,
  depicted in the right diagram in Figure \ref{f:subdivision}. We
  again interpolate $v$ in the intersection of $\mathcal
  C_{j+1}$-carriers of $a$ and $b$.
\end{enumerate}

Whenever we apply Corollary \ref{subtracks} it may happen that
conclusion (1) occurs. In that case we stop the process and do not
attempt to define a $\mathcal C_{j+1}$-good sequence.

To the $\mathcal C_j$-good sequence ${\bf a}_j=a_0,a_1,\dots,a_m$
constructed in this way we will associate a {\it dimension sequence}
$D({\bf a}_j)$ inductively. This is a sequence of nonnegative integers
$d_0,d_1,\dots,d_m$ with the requirement that the dimension of
$Carr_j(a_i)$ is $\leq d_i$. It is also constructed inductively. For
the initial sequence we take the dimensions of the $\mathcal C_0$-carriers.
Inductively, we extend the dimension sequence. For each curve $x$ that
is inserted when extending the sequence from ${\bf a}_j$ to ${\bf
  a}_{j+1}$ define the corresponding integer as the dimension of
$Carr_{j+1}(x)$. For curves that were part of the sequence ${\bf a}_j$
leave the value unchanged. Thus the number associated to a curve in
the sequence is the dimension of its carrier when the curve first
appeared. The dimension of the carrier of a curve may decrease, but
the value in the dimension sequence is unchanged.

The following proposition summarizes the essential features of the
construction. 

\begin{prop}
Suppose that a curve $x$ got inserted between the curves $a,b$ in a
$\mathcal C_j$-good sequence ${\bf a}_j$.
\begin{enumerate}[(i)]
\item The value of the dimension sequence at $x$ is strictly less than
  at both $a$ and $b$.
\item If $d(a,b)\leq C_i$ for some $i$ then $d(a,x)\leq C_{i+2}$ and
  $d(b,x)\leq C_{i+2}$.
\end{enumerate}
\end{prop}

In (ii) we may be applying Corollary \ref{subtracks} twice, and this
is why the conclusion involves $C_{i+2}$.

The following lemma can be proved by a straightforward induction on
$n$. We will apply it to dimension sequences.

\begin{lemma}\label{sequence bound}
Let ${\bf D}_i=(x_{i0},x_{i1},\dots,x_{ij_i})$, $i=0,1,2,\dots,N$ be a
sequence of finite sequences of nonnegative integers. Assume the
following:
\begin{enumerate}[(a)]
\item ${\bf D}_0=(x_{00},x_{01})$ has length 2 and
  $x_{00},x_{01}\leq n$,
\item for $i\geq 0$ the sequence ${\bf D}_{i+1}$ is obtained from
  ${\bf D}_i$ by inserting between some consecutive terms a
  nonnegative integer strictly smaller than each of the two terms.
\end{enumerate}
Then $j_N\leq 2^{n}$.
\end{lemma}

For example, $33, 323, 31213, 301020103$ is such a sequence with
$n=3$, $N=3$, $j_3=8$.

\begin{prop} \label{p:process}
For every $C>0$ there is $C'=C'(C,\Sigma)$ so that the following
holds.  Let $\mathcal C_j$ be an excellent sequence of cell structures
with all cells (i.e. their vertex cycles) at distance $\leq K$ from
$*$. Suppose $A,B$ are two cells in $\mathcal C_j$. If $a\in int(A)$
and $b\in int(B)$ and $d(a,b)\leq C$ then either
\begin{itemize}
\item $d(*,a),d(*,b)\leq K+C'$, or
\item there is a curve $c$ contained in a cell of $\mathcal C_j$ which is
  contained in a face of each $A$, $B$ such that
  $d(a,c),d(b,c)\leq C'$.
\end{itemize}
\end{prop}

\begin{proof}
First assume that $a,b$ is a $\mathcal C_0$-good sequence. We set
$C_0=C$ and define $C_i$ inductively as above.
Run the process starting with $a,b$. There are now two possibilities.

{\it Case 1.} The process produces a $\mathcal C_j$-good sequence
$a=a_0,a_1,\cdots,a_N$. From Lemma \ref{sequence bound} we see that
$N\leq 2^n$ where $n=\dim \mathcal{ML}$. Thus there were at most
$2^n-1$ insertions and this implies that $d(a_i,a_{i+1})\leq
C_{2(2^n-1)}$ for any two consecutive curves $a_i,a_{i+1}$. 

The sequence of $\mathcal C_j$-carriers $Carr_j(a_i)$, $i=0,1,\dots$
either increases or decreases (or stays the same) at every step. We
now modify the sequence, by ``pushing the peaks down'' so that an
initial part of the sequence of carriers is nonincreasing, and the
rest is nondecreasing. Let $a_i,a_{i+1},\dots,a_k$ be a subsequence of
consecutive curves such that $$Carr_j(a_i)\subsetneq
Carr_j(a_{i+1})=Carr_j(a_{i+2})=\dots=Carr_j(a_{k-1})\supsetneq
Carr_j(a_k)$$ First we pass to the length 3 subsequence
$a_i,a_{i+1},a_k$. The distance $d(a_{i+1},a_k)\leq 2^nC_{2(2^n-1)}$
(we are happy with very crude estimates), so applying Corollary
\ref{subtracks}(I) we find a curve $x$ with $Supp_j(x)\subset
Carr_j(a_i) \cap Carr_j(a_k)$ and with $d(a_i,x),d(a_k,x)\leq
C_{2(2^n-1)+2}$. (If conclusion (1) occurs see Case 2.) Continuing in
this way produces the desired sequence. The number of steps that
consist of pushing the peaks is bounded, e.g. by $n2^n$, so at the end
the distance between any two consecutive curves is bounded by
$C_{2(2^n-1)+2n2^n}$. Finally, pass to a length 3 sequence $a,c,b$
where $c$ has the minimal carrier, and set $C'=2^nC_{2(2^n-1)+2n2^n}$.

{\it Case 2.} At some stage in the process, when applying Corollary
\ref{subtracks}, conclusion (1) occurs. This applies also to the part
of the procedure when we push the peaks down. Thus we have a sequence
$a=a_0,a_1,\dots,a_N=b$ with $N\leq 2^n$, $d(a_i,a_{i+1})\leq
C_{2(2^n-1)+2n2^n}$, and for some $i$ we have $d(*,a_i)\leq
C_{2(2^n-1)+2n2^n+1}$. This implies by the triangle inequality that 
$$d(*,a),d(*,b)\leq C_{2(2^n-1)+2n2^n+1}+2^nC_{2(2^n-1)+2n2^n}$$
and we may take $C'$ to be this bound.

Finally, consider the general case when $a,b$ is not a $\mathcal
C_0$-good sequence. Let $A=Carr_0(a)$, $B=Carr_0(b)$. Thus
$A=V(\alpha), B=V(\beta)$ for certain tracks $\alpha,\beta$. Lemma
\ref{2 faces new} gives that either $d(*,a),d(*,b)$ are uniformly
bounded, as functions of $\Sigma$ and $C$, or there is a curve $c\in
A\cap B$ within uniform distance, call it $C_0$, from $a,b$. (Note
here that since $\mathcal C_0$ is a fixed cell structure, the
intersection number between any two tracks defining it is uniformly
bounded, so Lemma \ref{2 faces new} applies uniformly.) Thus
$a,c,b$ is a $\mathcal C_0$-good sequence and the procedure above
proves the statement.
\end{proof}

\section{Capacity dimension of $\mathcal{EL}$}

\subsection{Capacity dimension}

Let $(Z,\rho)$ be a metric space. The notion of capacity dimension of $Z$ was
introduced by Buyalo in \cite{buyalo}. One of several possible
equivalent definitions is the following, see \cite[Proposition
  3.2]{buyalo}. We also note that for bounded metric spaces, such as
the boundary of a hyperbolic space with visual metric, capacity
dimension agrees with the Assouad-Nagata dimension. See \cite{nagata}.

\begin{definition}
The capacity dimension of a metric space $Z$ is the infimum of all
integers $m$ with the following property. There exists a constant
$c>0$ such that for all sufficiently small $s>0$, $Z$ has a
$cs$-bounded covering with $s$-multiplicity at most $m+1$.
\end{definition}

The covering $\mathcal L$ is $cs$-bounded if all elements have
diameter $<cs$ and the $s$-multiplicity of $\mathcal L$ is $\leq m+1$
if every $z\in Z$ is at distance $<s$ from at most $m+1$ elements of
$\mathcal L$.

We will produce covers that resemble cell structures and whose
thickenings resemble handle decompositions. It is more convenient here
to index the handles starting with 1, rather than with 0. We will use
following form of the definition of capacity dimension.

\begin{prop}\label{capdim}
Suppose that there is a constant $c>10$ such that for all sufficiently
small $s>0$ there is a cover $\mathcal K$ of $Z$ with the following
properties:
\begin{itemize}
\item The collection $\mathcal K$ is the disjoint union of
  subcollections $\mathcal K_1,\mathcal K_2,\cdots,\mathcal K_{m+1}$.
\item The diameter of any set in $\mathcal K$ is $\leq s$.
\item If $A,B\in\mathcal K_i$ are distinct elements in the same
  subcollection and if $a\in A$, $b\in B$, $\rho(a,b)<s/c^{3i-1}$ then
  there is some $e\in E\in \mathcal K_k$ for some $k<i$ so that
  $\rho(a,e)<s/c^{3i-2}$ and $\rho(b,e)<s/c^{3i-2}$.
\end{itemize}
Then the capacity dimension of $Z$ is at most $m$.
\end{prop}

\begin{proof}
Inductively on $i$, for each $K\in \mathcal K_i$ define the associated
``handle'' $$H(K)=N_{s/c^{3i}}(K)-\cup_{K'\in \mathcal K_k, k<i}
H(K')$$ It is clear that the collection of all handles forms a cover
of $Z$ and that the diameter of each element is bounded by
$s+2s/c^3$. We will argue that the $s/c^{3m+4}$-multiplicity of the
cover is $\leq m+1$. Suppose $z\in Z$ is at distance $<s/c^{3m+4}$
from $m+2$ handles. Then two of the handles have the same index, say
$H(A),H(B)$ with $A,B\in\mathcal K_i$. Thus we have $a_0\in H(A)$ and
$b_0\in H(B)$ with $\rho(a_0,b_0)<2s/c^{3m+4}$. Choose $a\in A$ and
$b\in B$ with $\rho(a,a_0)<s/c^{3i}$ and $\rho(b,b_0)<s/c^{3i}$. Then
$\rho(a,b)<2s/c^{3m+4}+2s/c^{3i}<s/c^{3i-1}$. By assumption there is
$e\in K\in\mathcal K_k$ with $k<i$ and with $\rho(a,e)<s/c^{3i-2}$,
$\rho(b,e)<s/c^{3i-2}$. Thus
$\rho(a_0,e)<s/c^{3i-2}+s/c^{3i}<s/c^{3i-3}\leq s/c^{3k}$, so $a_0\in
H(K)$ or it belongs to a lower index handle, and similarly for
$b_0$. This is a contradiction since $H(A)$ and $H(B)$ are disjoint
from lower index handles.
\end{proof}

\begin{figure}
  \begin{center}
    \begin{tikzpicture}[scale=0.6,y=-1cm]
\sf
\draw[semithick,black] (5.8,5.5) circle (1.42222cm);
\draw[semithick,black] (12.30444,5.5) circle (1.42222cm);
\draw[semithick,black] (12.30444,12.00889) circle (1.42222cm);
\draw[semithick,black] (5.91111,12.00889) circle (1.42222cm);
\draw[black] (12.4,12.1) -- (12.4,5.5) -- (5.8,5.5) -- (5.8,12.1) -- cycle;
\draw[black] (5.8,5.5) -- (12.4,12.1);
\draw[semithick,black] (12.8,6.8) -- (12.8,10.7);
\draw[semithick,black] (7,6.2) -- (11.6,10.8);
\draw[semithick,black] (6.6,6.7) -- (11.1,11.2);
\draw[semithick,black] (7.3,11.7) -- (10.9,11.7);
\draw[semithick,black] (7.2,12.5) -- (11,12.5);
\draw[semithick,black] (6.3,6.8) -- (6.3,10.6);
\draw[semithick,black] (5.4,6.8) -- (5.4,10.7);
\draw[semithick,black] (7.1,5.1) -- (10.9,5.1);
\draw[semithick,black] (7.1,5.9) -- (10.9,5.9);
\draw[semithick,black] (12,6.9) -- (12,10.6);

\end{tikzpicture}%

\hskip 3cm
\begin{tikzpicture}[scale=0.8,y=-1cm]
\sf
\draw[black] (6.36176,8) +(-66:2.83824) arc (-66:66:2.83824);
\draw[black] (13.13824,8) +(-114:2.83824) arc (-114:-246:2.83824);
\draw[black] (7.5,5.4) -- (12,5.4);
\draw[black] (7.5,10.6) -- (12,10.6);

\end{tikzpicture}%

\caption{The handle decomposition associated to a triangulation. Cells
  shaped like the one pictured on the right will result in handle
  decompositions with bad Lebesgue number.}
\label{bad}
\end{center}
\end{figure}
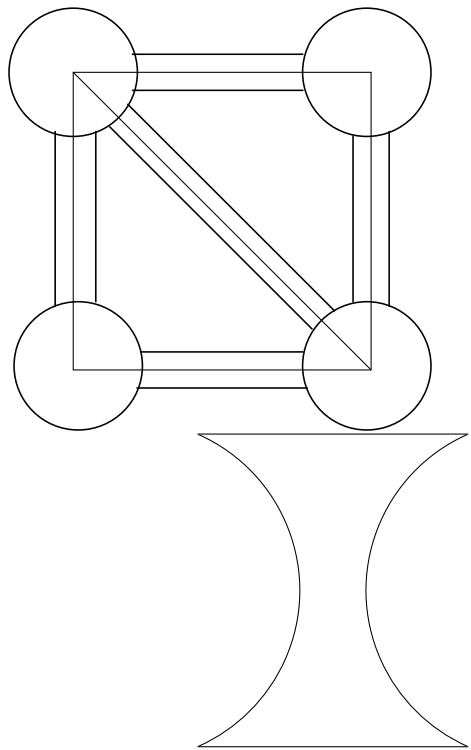

\begin{remark}
In section \ref{the cover} we will show that the cover of $\mathcal
{EL}$ induced by $P_\infty(\sigma)$ as $V(\sigma)$ range over cells in
$\mathcal C_j$ with $\sigma$ large satisfy the assumptions of
Proposition \ref{capdim}. For example, shapes like in Figure \ref{bad}
are ruled out by Proposition \ref{p:process}.
\end{remark}

We also have the following general fact. We thank Vera Toni\'c for
pointing it out to us.

\begin{prop}[\cite{nagata}]\label{vera}
Suppose $Z$ is written as a finite union of (closed) subsets $Z_i$. If
$\capdim Z_i\leq n$ for all $i$, then $\capdim Z\leq n$.
\end{prop}

\subsection{Visual size and distance}

Recall that a metric $\rho$ on the boundary $Z$ of a
$\delta$-hyperbolic space $X$ is said to be {\it visual} if there is a
basepoint $*\in X$ and constants $a>1$ and $c_1,c_2>0$ such that
$$c_1a^{-(z|z')}\leq \rho(z,z')\leq c_2a^{-(z|z')}$$
for all $z,z'\in Z$, where $(\cdot|\cdot)$ denotes the Gromov product
$$(x|x')=\frac 12 (d(*,x)+d(*,x')-d(x,x'))$$
on $X$, extended naturally to $Z$. See
\cite[Ch.~7]{ghys} for more details and for the construction of visual metrics.

Also recall that, in a $\delta$-hyperbolic space, the Gromov product
$(a|b)$ is, to within a uniform bound that depends on $\delta$, the
distance between the basepoint $*$ and any geodesic $[a,b]$. The same
is true when $a,b$ are distinct points at infinity. We may also
replace geodesics $[a,b]$ with quasi-geodesics, but then the uniform
bound depends also on quasi-geodesic constants.

A $\delta$-hyperbolic space $X$ is {\it visual} \cite{bonk-schramm}
for some (every) basepoint $x_0$ there exists $C>0$ such that for
every $x\in X$ there is a $(C,C)$-quasi-geodesic ray in $X$ based at
$x_0$ and passing through $x$. Equivalently, $X$ is
the coarse convex hull of the boundary $\partial X$. Any
$\delta$-hyperbolic space whose isometry group acts coboundedly and
that contains a biinfinite quasi-geodesic line is visual. Thus a curve
complex is visual.

\begin{thm}\cite{buyalo}
Let $X$ be a visual $\delta$-hyperbolic metric space and $Z$ its Gromov
boundary endowed with a visual metric. Then 
$$\asdim(X)\leq \capdim(Z)+1$$
\end{thm}

\begin{lemma}\label{visual size}
Assume that $\sigma$ is a large track obtained from $\tau$ by a sequence of
splittings.  Using $B(\tau)$ as a basepoint, the visual diameter of
$P_\infty(\sigma)$ is $a^{-d(*,B(\sigma))}$, to within a bounded factor.
\end{lemma}

\begin{proof} This follows from Lemmas \ref{3.4} and \ref{B}.
\end{proof}

\begin{prop}
Let $\mathcal C_j$ be an excellent sequence of cell structures
obtained by splitting tracks. For all sufficiently
large constants $c$ (depending only on $\Sigma$) the following holds
for all sufficiently small $s>0$. Suppose for a certain $j$ the visual
diameter of each $P_\infty(\sigma)$ is $>s$, where $\sigma$ ranges
over all maximal train tracks such that $V(\sigma)\in \mathcal C_j$. Then
\begin{itemize}
\item The visual diameter of each $P_\infty(\sigma)$ is $>s/c$ for
  every large track $\sigma$ determining a cell $V(\sigma)$ in
  $\mathcal C_j$.
\item Suppose cells $V(\sigma)$ and $V(\sigma')$ in $\mathcal C_j$
  are distinct with both $\sigma,\sigma'$ large. Suppose
  $a\in P_\infty(\sigma)$, $b\in P_\infty(\sigma')$, and
  $\rho(a,b)<s/c^2$. Then there is a cell $V(\mu)\subset
  V(\sigma)\cap V(\sigma')$ and there is a $e\in P_\infty(\mu)$  
  so that $\rho(a,e)<s/c$, $\rho(b,e)<s/c$.
\end{itemize}
\end{prop}

\begin{proof}
The first bullet follows from Lemma \ref{visual size} and Proposition
\ref{distance}.

To prove the second bullet we use Proposition
\ref{p:process}. Consider a quasi-geodesic ray from $*$ to $a$ that
passes through $B(\sigma)$ and between $B(\sigma)$ and $a$ stays in
$S(\sigma)$. By Lemmas \ref{3.3} and \ref{3.4} we may assume that these
rays are uniformly quasi-geodesic. Likewise, construct such a ray from $*$
to $b$. Also choose a uniform quasi-geodesic from $a$ to $b$. We
now
have a uniformly thin triangle with two vertices at infinity. Choose
$a'$ on the first ray and $b'$ on the second ray, just past the thick
part viewed from $*$. Thus $d(a',b')$ is uniformly bounded, say by
$C$. Also note that $d(*,a'),d(*,b')$ is a definite amount larger than
$d(*,B(\sigma)),d(*,B(\sigma'))$ by the assumption that
$\rho(a,b)<s/c$ with $c$ sufficiently large. In particular, $a'\in
S(\sigma)$, $b'\in S(\sigma')$. For this $C$ Proposition
\ref{p:process} provides a constant $C'$. Now the first bullet in the
conclusion of Proposition \ref{p:process} cannot hold if $c$ is
sufficiently large. Therefore, there is some $e'\in P(\mu)\in \mathcal
C_j$ with $d(a',e'),d(b',e')\leq C'$ and with $\dim V(\mu)<\dim
V(\sigma)=\dim V(\sigma')$. Again using Lemmas \ref{3.3} and \ref{3.4}
construct a uniform quasi-geodesic ray from $*$ through $e'$ to some
$e\in P_\infty(\mu)$. This $e$ satisfies the conclusions.
\end{proof}

\subsection{The cover}\label{the cover}

It is well known that $\mathcal {PML}$ can be covered by finitely many
sets of the form $P(\tau)$ for $\tau$ a large train track (for a
concrete cover see \cite{ph}). Thus $\mathcal {EL}$ is finitely
covered by sets of the form $P_\infty(\tau)$.
In view of Proposition \ref{vera} we
need to find an upper bound to $\capdim P_\infty(\tau)$.

Here we fix a large birecurrent train track $\tau$ and describe a
cover of $Z=P_\infty(\tau)$ that satisfies Proposition \ref{capdim}
for a certain $m$ depending on $\Sigma$ and for small $s>0$.

The dimensions of cones $V(\sigma)$ for large train tracks
$\sigma\subset\Sigma$ belong to a certain interval $[A,A+K]$ that
depends on $\Sigma$. We put $m=K$. We also fix a large constant $c$.

Now fix a small $s>0$ and start with the standard cell structure on
$V(\tau)$. This is $\mathcal C_0$. Now suppose
$\mathcal C_j$ has been constructed and the visual size of each
$P_\infty(\sigma)$ for a top dimensional cell $V(\sigma)\in\mathcal
C_j$ is $>s/c$. Enumerate all top dimensional cells
$V(\sigma)\in\mathcal C_j$ such that the visual size of
$P_\infty(\sigma)$ is $\geq s$ and also enumerate all large branches in
the corresponding tracks $\sigma$. Then construct $\mathcal
C_{j+1},\cdots,\mathcal C_k$ by splitting along these branches, in any
order. We call this collection of splits a {\it multisplit}.

This gives an infinite excellent sequence. Note that once some
$P_\infty(\sigma)$ (with $V(\sigma)$ maximal) reaches visual size $<s$ at
the end of a multisplit, it never gets subdivided again (see Lemma
\ref{3.8}). Coarsely, reaching a certain visual size is equivalent to
$B(\sigma)$ reaching a certain distance from the basepoint
$*$ (see Lemma \ref{visual size}).

\begin{lemma}
Let $\lambda$ be a filling lamination, and for every $j$ let
$E_j=V(\sigma_j)$ be the cell in $\mathcal C_j$ that contains
$\lambda$ in its interior. Then the sequence $E_j$ eventually
stabilizes.
\end{lemma}

\begin{proof} We argue by contradiction.
From Proposition \ref{distance} we have that
$\sigma_{j+1}\sscm\sigma_j$.  Let $a_j$ be a vertex cycle of
$\sigma_j$. Then we may assume, perhaps after a subsequence, that
$a_j\to\lambda'$ in $\mathcal{PML}$ and the lamination $\lambda'$ is
necessarily carried by all $\sigma_j$. By Lemma \ref{split to full
  carry} for large $j$, $\sigma_j$ will fully carry $\lambda$, so by
Lemma \ref{splitting sequence} we have $[\lambda']=[\lambda]$. By an
argument of Kobyashi (see p. 124 of \cite{MM}) the sequence $a_j$ goes
to infinity in the curve complex, so the visual size of $E_j$ goes to
0 by Lemma \ref{visual size}. But in the construction of $\mathcal
C_j$ the visual diameter of all top dimensional cells is bounded below
and by Proposition \ref{distance} this bounds below the visual
diameter of all cells, giving a contradiction.
\end{proof}

We let the cover $\mathcal K$ consist of the sets of the form
$P_\infty(\sigma)$ such that $V(\sigma)$ is a stable cell. We
partition the sets in $\mathcal K$ according to the dimension of the
cell.

\begin{thm}
$$\capdim(EL)\leq K(\Sigma)$$ where $K=K(\Sigma)$ is the smallest
  integer such that every recurrent, transversely recurrent,
  large track $\sigma$ on $\Sigma$ has 
  $\dim P(\sigma)\in [A,A+K]$ for some $A=A(\Sigma)$.
\end{thm}

\begin{proof}
We only need to argue that the cover $\mathcal K$ satisfies the
conditions of Proposition \ref{capdim}. This is clear from the
construction and Propositions \ref{distance} and \ref{p:process}
applied to $\mathcal C_j$ for large $j$.
\end{proof}

\begin{cor}
$$\asdim(\mathcal C(\Sigma))\leq K(\Sigma)+1$$
\end{cor}

\begin{example}
One can see easily what happens in the case of the punctured
torus. Then $\mathcal{FPML}=\mathcal{EL}$ is homeomorphic to the set
of irrational numbers, or equivalently to $\Z^{\infty}$. The visual
metric is complete, and the cover $\mathcal K$ constructed above will
be infinite and will consist of pairwise disjoint sets, all of the
same index, and all of comparable sizes. For example, consider a
standard track that supports laminations whose slope is in the
interval $[1,\infty]$. Splitting produces two tracks, one carrying
laminations in the interval $[1,2]$ and the other in the interval
$[2,\infty]$. We can take the curve with slope $\infty$ as the
basepoint and agree to stop subdividing when the distance from
$\infty$ to $B(\sigma)$ is $>0$, i.e. when $\infty$ is no longer
carried by $\sigma$. Thus we stop splitting the track carrying
$[1,2]$ and we split the other track. We get tracks carrying $[2,3]$
and $[3,\infty]$. Continuing in this way we get an infinite cover
$P_\infty(\sigma_n)$ of $\mathcal{EL}$ where $\sigma_n$ carries ending
laminations with slope in $[n,n+1]$.
\end{example}

\begin{remark}
There are two other closely related notions to asymptotic
dimension. In the {\it linearly controlled} asymptotic dimension, or
the {\it asymptotic Assouad-Nagata} dimension ${\rm \ell-asdim}$ one
insists on the linear control on the size of the cover.
Also, say that $eco-dim(X)= n$ if $X$
quasi-isometrically embeds in a product of $n$ trees and $n$ is
smallest possible. Then there is a chain of inequalities
$$\asdim\leq {\rm \ell-asdim}\leq {\rm eco-dim}$$
for any metric space, and Buyalo shows in \cite{buyalo2} that when $X$
is $\delta$-hyperbolic then
$$eco-dim(X)\leq \capdim\partial X+1$$
See also the discussion in \cite{mac-sisto}. Therefore our arguments
also give the same bound on $l-asdim$ and $eco-dim$ for $\mathcal
C(\Sigma)$. Previously, Hume observed that $eco-dim(\mathcal
C(\Sigma))<\infty$ \cite{hume}.
\end{remark}

\begin{appendices}
\section{Train tracks and full dimension paths}\label{splitting paths}

A {\em splitting path} is a legal embedded path in a thickened train
track that begins and ends at a cusp.
See Figure \ref{splitting path}.

\begin{figure}
  \begin{center}
\includegraphics[scale=0.4]{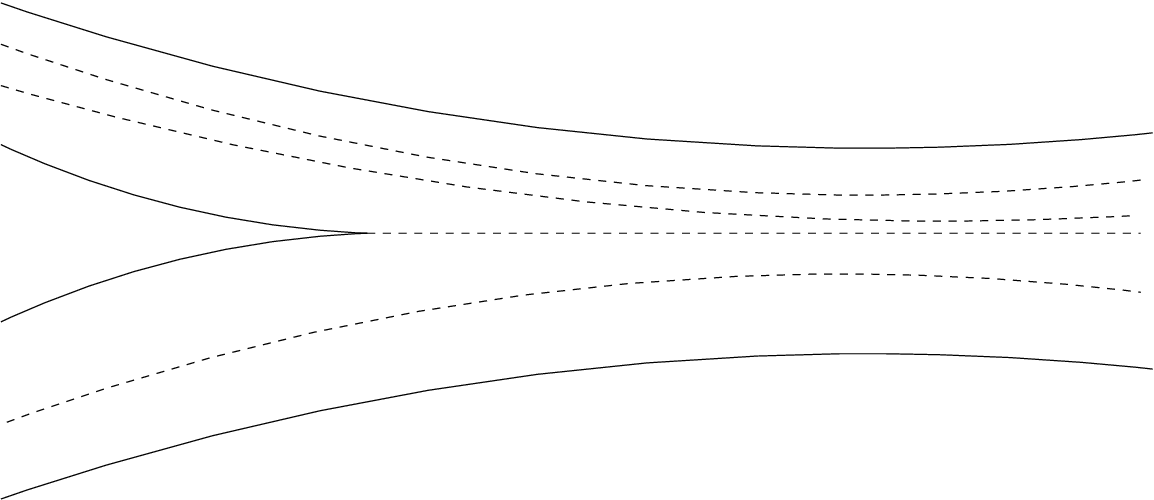}
\caption{The dashed line is the splitting path. It must start and end at a cusp in the thickened track.}
\label{splitting path}
\end{center}
\end{figure}

If $\tau$ is a recurrent train track and
$\theta$ is the central split along the splitting path then $\theta$
will have one or two connected components and a total of three less
branches and two less switches then $\tau$. By Lemma \ref{dimension
  count} we then have:
\begin{lemma}\label{fd splits}
Either $\dim V(\theta) < \dim V(\tau)$ or $\dim V(\theta) = \dim
V(\tau)$. If $\dim V(\theta) = \dim
V(\tau)$ then one of the following holds.
\begin{enumerate}[(1)]
\item $\tau$ is non-orientable and $\theta$ is connected and orientable,

\item $\tau$ is orientable and $\theta$ has two components (both necessarily orientable) or

\item $\tau$ is non-orientable and $\theta$ has one orientable
  component and one non-orientable component.
\end{enumerate}
\end{lemma}

If $\dim V(\theta) = \dim V(\tau)$ we say that the splitting path is a
{\em full dimension splitting path} or fd-path. While a splitting path
will be embedded in the thickened train track, in the actual train
track it may cross a single branch multiple times. However, an fd-path
can cross any branch at most twice and this strong restriction implies
that there is a uniform bound on the number of fd-paths in a given
track.
\begin{lemma}\label{two-to-one}
An fd-path of types (1) or (2) crosses each branch of $\tau$ at most once and an fd-path of type (3) crosses each branch at most twice.
\end{lemma}

\begin{proof}
Orient the splitting path. If the path crosses a branch more than
once, we examine two consecutive strands of the path in the branch, as
seen crossing the branch transversally.
\begin{itemize}
\item If the two strands have the same orientation then $\theta$ will
  be connected and we must be in case (1). Then $\tau$ will be
  non-orientable and $\theta$ will be orientable so we can choose an
  orientation for $\theta$. On opposite sides of the splitting path
  the orientation of $\theta$ cannot agree or $\tau$ would be
  orientable. Using the orientation of the splitting path (and the
  surface) we can assume that to the right of the splitting path the
  orientation of $\theta$ and the splitting path agree while to the
  left they are opposite of each other. However, this is not possible
  if there are two consecutive strands in the same branch with the
  same orientation. Therefore we can never have consecutive branches
  with the same orientation.

\item If consecutive branches (seen transversally) in the same branch have opposite
  orientation then the component of $\theta$ between the two strands
  will be non-orientable and we must be in case (3). If the splitting
  path crosses the same branch three or more times then orientation of
  consecutive branches will always be opposite so both components of
  $\theta$ will be non-orientable, a contradiction.
\end{itemize}
\end{proof}

\begin{lemma}\label{fd-paths lift}
Let $\tau_1$ and $\tau_2$ be recurrent train tracks with
$\tau_2\sscm\tau_1$ and $\dim V(\tau_1) = \dim V(\tau_2)$. Let
$\theta_2$ be the central split of $\tau_2$ along a splitting
path. Then there is a splitting path in $\tau_1$ with central split
$\theta_1$ such that $V(\theta_2) = V(\theta_1) \cap V(\tau_2)$.
\end{lemma}

\begin{proof}
As usual via induction we can reduce this to the case when
$\tau_2\sscm\tau_1$ is a single move. In $\tau_2$ the splitting path
starts and ends at large half branches. If these large half branches
and their adjacent half branches are not active in the move then the
composition of the splitting path with the carrying map is a splitting
path. If not the carrying map is a ``fold'' along the switch and, in
$\tau_1$, we extend the path along the fold.

If the move is a right or left split along a large branch $b$ in
$\tau_1$ then $\theta_1 = \theta_2$ if the path crosses $b$ in
$\tau_2$. If not that $\theta_2$ is a split of $\theta_1$ along
$b$. If the move is a central split then $\theta_2$ is a central split
of $\theta_1$. \end{proof}

One consequence of the existence of fd-paths is that there are large train tracks that do not fully carry any lamination. We will show that this is the only obstruction.

\begin{lemma}\label{split to full carry}
Let $\tau$ be a recurrent train track and let $\lambda$ be a lamination in the interior of $V(\tau)$.
Then there exists a recurrent train track $\sigma$ with $\sigma\scm\tau$ and $\lambda$ is fully carried by $\sigma$.
\end{lemma}

\begin{proof}
By \cite[Proposition 8.9.2]{Thurston:book:GTTM} there exists a recurrent (in fact
birecurrent) train track $\tau'$ that fully carries $\lambda$. Note
that while fully carrying is not discussed in this proposition, one
sees that if the $\epsilon$ in the construction is chosen to be
sufficiently small then the track will be fully carrying. Then, as in Proposition \ref{splitting sequence}, we use  \cite[Theorem 2.3.1]{ph} to find a train track $\sigma$ that carries $\lambda$ with $\sigma\scm\tau'$ and $\sigma\scm\tau$. As $\tau'$ fully carries $\lambda$ so must $\sigma$. 
\end{proof}

Observe that if $\sigma\sscm\tau$ and $\dim V(\sigma) < \dim V(\tau)$ then there will be a hyperplane $P$ defined by equations that have rational coefficients and such that $P$ has positive co-dimension and $V(\sigma) \subset P\cap V(\tau)$.

\begin{prop}\label{ending laminations}
Let $\tau$ be a recurrent large train track and assume that $\lambda
\in V(\tau)$ is not contained in a rational hyperplane of positive
co-dimension. Then either $\tau$ contains an fd-path with central
split $\theta$ and $\lambda \in V(\theta)$ or $\lambda$ is an ending
lamination.
\end{prop}

\begin{proof}
By Lemma \ref{split to full carry} we can find a recurrent train $\sigma$ such that $\sigma\scm\tau$ and $\sigma$ fully carries $\lambda$.
If
$\tau$ fully carries $\sigma$ then it also fully carries $\lambda$ and
$\lambda$ must be an ending lamination. If not a central split must
occur in the sequence $\sigma\scm\tau$. Let $\tau_1$ be the track in
the sequence that occurs just before the first central split and let $b$ be the large branch where the central split occurs. Then $\dim V(\sigma) = \dim V(\tau)$ for if not $\lambda$ we be contained in a rational hyperplane of positive co-dimension. Therefore the large branch $b$ is an fd-path of length one. Let 
$\theta_1$ the central split of $\tau_1$ along the fd-path $b$. By Lemma
\ref{fd-paths lift} there exists an fd-path in $\tau$ with central
split $\theta$ such that $V(\theta_1) = V(\tau) \cap V(\theta)$ so
$\lambda \in V(\theta)$.  \end{proof}

To summarize, in the presence of fd-paths it is generally not true
that for a large track $\sigma$ generic points of $V(\sigma)$ (those in the
complement of rational hyperplanes) represent ending laminations, but
this will be true in the complement of a finite collection of subcells
of $V(\sigma)$. The subcells could cover $V(\sigma)$, but for example
if $V(\sigma)$ contains one ending lamination, it contains infinitely
many. 

\begin{cor}\label{dense ending laminations}
Let $\tau$ be a birecurrent train track and $a \in S(\tau)$ a simple
closed curve. There exists a $C = C(\Sigma)$ such that either
$d(B(\tau), a) \le C$ or there exists a sequence $\lambda_i \in
P_\infty(\tau)$ such that $a$ is contained in the Hausdorff limit of
the $[\lambda_i]$.
\end{cor}

\begin{proof}
By Lemma \ref{two-to-one}, $\tau$ contains finitely many
fd-paths. Assume there are $k\ge 0$ such paths. We begin by splitting
on each of these paths to obtain $k$ new tracks which we label
$\theta_1,\dots, \theta_k$. If $a$ is in the complement of $\cup
V(\theta_i)$ then the corollary follows from Proposition \ref{ending
  laminations} applied to a sequence of laminations in $V(\tau)$
converging to $a$ and not contained in proper rational planes. If not
$a \in S(\theta_i)$ for some $i$. If $\theta_i$ is small then
$d(B(\theta_i), a) \le 2$ and $d(B(\theta_i), B(\tau))$ is uniformly
bounded since an fd-path is at worst two-to-one by Lemma
\ref{two-to-one}. Therefore $d(B(\tau),a)$ is uniformly bounded.

If $\theta_i$ is large then in it is connected and, by Lemma \ref{fd
  splits}, orientable. As in the previous paragraph we split along all
fd-paths to get a collection of tracks $\theta^i_1, \dots,
\theta^i_{k_i}$. Since $\theta_i$ is orientable, Lemma \ref{fd splits}
implies that the $\theta^i_j$ are disconnected and hence small. If $a
\in \theta^i_j$ for some $j$ then $a$ is uniformly close to
$B(\theta^i_j)$, and therefore to $B(\tau)$, as above. If not, we
again apply Proposition \ref{ending laminations}.
\end{proof}

\end{appendices}
\bibliography{./ref}
\end{document}